\documentclass[12pt]{amsart}%
\usepackage{amsfonts}
\usepackage{amsmath}
\usepackage{amssymb}
\usepackage{amsthm}
\usepackage{newlfont}
\usepackage{amsmath}
\usepackage{graphicx}
\usepackage{txfonts}
\usepackage{mathtools}
\usepackage[mathscr]{euscript}%
\providecommand{\U}[1]{\protect\rule{.1in}{.1in}}
\providecommand{\U}[1]{\protect\rule{.1in}{.1in}}
\newtheorem{theorem}{Theorem}[section]

\newtheorem{corollary}[theorem]{Corollary}

\newtheorem{definition}[theorem]{Definition}

\newtheorem{lemma}[theorem]{Lemma}

\newtheorem{proposition}[theorem]{Proposition}

\numberwithin{equation}{section}
\begin{document}
\title[Compactification of Hamiltonian stationary Lagrangian submanifolds]{Compactification of the space of Hamiltonian stationary Lagrangian
submanifolds with bounded total extrinsic curvature and volume}
\author{Jingyi Chen}
\address{Department of Mathematics\\
The University of British Columbia\\
Vancouver, B.C. V6T1Z2, Canada}
\email{jychen@math.ubc.ca}
\author{Micah Warren}
\address{Department of Mathematics\\
University of Oregon, Eugene, OR 97403, U.S.A.}
\email{micahw@uoregon.edu}
\thanks{The first author was partially supported by an NSERC Discovery Grant
(22R80062) and a grant (No. 562829) from the Simons Foundation. }
\date{}

\begin{abstract}
For a sequence of immersed connected closed Hamiltonian stationary Lagrangian
submaniolds in $\mathbb{C}^{n}$ with uniform bounds on their volumes and the
total extrinsic curvatures, we prove that a subsequence converges either to a
point or to a Hamiltonian stationary Lagrangian $n$-varifold locally uniformly
in $C^{k}$ for any nonnegative integer $k$ away from a finite set of points,
and the limit is Hamiltonian stationary in ${\mathbb{C}}^{n}$. We also obtain
a theorem on extending Hamiltonian stationary Lagrangian submanifolds $L$
across a compact set $N$ of Hausdorff codimension at least 2 that is locally
noncollapsing in volumes matching its Hausdorff dimension, provided the mean
curvature of $L$ is in $L^{n}$ and a condition on local volume of $L$ near $N$
is satisfied.

\end{abstract}
\maketitle

\section{Introduction}

Compactness of stationary points of the volume functional, possibly under
various constraints, is useful in studying existence and regularity of the
critical points and their moduli space. For compactness of minimal surfaces,
Choi-Schoen demonstrated in their well-known work \cite{CSc} that Simons'
identity for the Laplacian of the second fundamental form \cite{Simons} can be
used to derive curvature estimates when the total extrinsic curvature over a
ball is small and then obtained higher order curvature estimates. This
influential technique now becomes standard when the Euler-Lagrange equation of
the volume related variational problem is of second order.


Hamiltonian stationary Lagrangian submanifolds in $\mathbb{C}^{n}$ are
critical points of the volume functional under Hamiltonian variations
$X=J\hbox{D}f$ for any compactly supported smooth function $f$ on
$\mathbb{C}^{n}$ \cite{MR1202805}. Any smooth Lagrangian submanifold in
$\mathbb{C}^{n}$ can be locally defined by a graph over a region $\Omega$ in a
Lagrangian tangent plane, in the form
\[
\Gamma_{u}=\left\{  \left(  x,Du(x)\right)  :x\in\Omega\right\}
\]
for some $u\in C^{\infty}(\Omega)$. If the Lagrangian phase
\begin{equation}
\theta=\sum_{\lambda_{j}\text{ eigenvalues of }D^{2}u}\arctan\lambda_{j}\text{
}\label{theta}%
\end{equation}
is constant, then the Lagrangian submanifold is volume minimizing among all
submanifolds in the same homology class, as shown in \cite{HL}. If the phase
$\theta$ is harmonic on $\Gamma_{u}$, that is,
\begin{equation}
\Delta_{g}\theta=0\label{hseq}%
\end{equation}
where $\Delta_{g}$ is the Laplace-Beltrami operator on $\Gamma_{u}$ for the
induced metric $g$, then $\Gamma_{u}$ is Hamiltonian stationary, and vice
versa (cf. \cite{MR1202805}, \cite[Proposition 2.2]{SW}). Equation
\eqref{hseq} is a fourth order nonlinear elliptic equation for the potential
function $u$. An important feature of the fourth order operator is its
decomposition into two second order elliptic operators, and this is the basis
for our curvature estimate and smoothness estimates, as already used in our
regularity theory on Hamiltonian stationary Lagrangian submanifolds
\cite{ChenWarren}.

In this paper, we prove a compactness result for closed immersed Hamiltonian
stationary Lagrangian submanifolds of $\mathbb{C}^{n}$ with uniform bound on
volume and total extrinsic curvature, namely, the $L^{n}$-norm of the second
fundamental form. For any sequence of such submanifolds, we show that a
subsequence converges, locally uniformly in every $C^{k}$-norm away from a
finite set of points, to an integral varifold which is Hamiltonian stationary
in an appropriate sense. So we can compactify the space of these submanifolds
by including Hamiltonian stationary integral $n$-varifolds with only point
singularities (immersed elsewhere) and the number of the singular points
bounded by a constant depending only on the upper bound of the total extrinsic
curvature. It is possible that the sequence converges to a point, such as
shrinking circles in the plane. This can be excluded by scaling volume to one,
while the total extrinsic curvature and being Hamiltonian stationary
Lagrangian are both scaling invariant, although the Hamiltonian isotopy
classes may change.


\begin{theorem}
\label{conv_thm} Suppose that $\{L_{i}\}$ is a sequence of connected
Lagrangian Hamiltonian stationary closed (compact without boundary) immersed
submanifolds of $\mathbb{C}^{n}$ with $0\in L_{i}$ and
\[
Volume(L_{i})<C_{1}\,\,\,\hbox{and}\,\,\int_{L_{i}}\left\vert A\right\vert
^{n}\emph{d}\mu_{L_{i}}<C_{2}.
\]
Then $L_{i}\subset B_{R_{0}}(0)\subset\mathbb{R}^{2n}$ for some $R_{0}%
(n,C_{1},C_{2})$. Moreover, there exists a subsequence of $\{L_{i}\}$ that
either converges to a point, or converges to a Hamiltonian stationary
Lagrangian varifold on $\mathbb{C}^{n}\backslash S$ for some finite set $S$ in
the $C^{k}$ topology on any compact subset of $B_{R_{0}}(0)\backslash S$. The
limiting varifold is supported, possibly with multiplicity, on an immersed
submanifold $L$. The closure $\overline{L}$ is Hamiltonian stationary in
$\mathbb{C}^{n}$ in the sense that the generalized mean curvature
$\mathcal{H}$ of the varifold $(\overline{L},\mu_{L})$ on $\mathbb{R}^{2n}$
exists and satisfies
\begin{equation}
\int_{\mathbb{R}^{2n}}\langle J\hbox{D}f,\mathcal{H}\rangle\,\emph{d}\mu
_{L}=0\label{HSInt}%
\end{equation}
for any $f\in C_{0}^{\infty}(\mathbb{R}^{2n})$. Also, $\overline{L}$ is connected.
\end{theorem}

We also obtain an extendibility result in Theorem \ref{N} which asserts that a
properly immersed Lagrangian submanifold $L$ that is Hamiltonian stationary in
${\mathbb{C}}^{n}\backslash N$ (i.e. for Hamiltonian vector fields supported
away from $N$) is Hamiltonian stationary in ${\mathbb{C}}^{n}$ (i.e. for all
compactly supported Hamiltonian vector fields), provided $N$ is a compact set
with finite $k$-dimensional Hausdorff measure which is locally $k$%
-noncollapsing, $k\leq n-2$, and the volume of $L\cap B_{r}(x)$ for $x\in N$
is dominated by a power of $r$ involving $n,k$. Local control on volume is
important for extension problems; our consideration is inspired by those for
extending minimal varieties (general dimension and codimension) across small
closed sets in \cite[Theorem 5.1, 5.2]{Harvey-Lawson-70}, also see
\cite{Chen-Li}. A special case of Theorem \ref{N}, namely, when $N$ is a
finite set of points, is used in concluding the limiting varifold in Theorem
\ref{conv_thm} is Hamiltonian stationary in $\mathbb{C}^{n}$. A removable
singularity theorem for Hamiltonian stationary Lagrangian \textit{graphs} was
proven in \cite{ChenWarren} under a weaker assumption.

There are two natural ways to give an immersed submanifold a varifold
structure. Denoting by $\mathscr{H}^{k}$ the $k$-dimensional Hausdorff
measure,\ \ any $\mathscr H^{k}$-measurable and rectifiable subset of
$\mathbb{R}^{2n}$ is associated with a varifold naturally \cite[p. 61]{Pitts},
by restricting the $\mathscr{H}^{k}$ measure to each approximate tangent
space. This takes into consideration only the point set of the image.
Meanwhile, the image of an immersion $\iota:M^{k}\rightarrow\mathbb{R}^{2n}$
is also associated naturally to a varifold by pushing forward the pulled-back
$\mathscr{H}^{k}$ measure. These two definitions differ only if the immersion
fails to be injective on non-negligible set. Here we take our sequence
$\{L_{i}\}$ to be smooth immersions, which puts us in the latter setting. This
latter definition may be more natural when studying sequences, flows or moduli
spaces of submanifolds, as it has the feature that the limit does not lose
mass as a varifold, so the weak convergence in the varifold topology is
faithful to the convergence in other natural topologies (for example $L^{2}$
or $W^{1,2}$ induced length-metric, see \cite[Section II]{R17}) that one may
place on a space of differentiable submanifolds. However, due to the
analyticity of solutions to \eqref{hseq}, we find (Proposition \ref{canonical}%
) that any Hamiltonian stationary immersion from a compact connected manifold
can be passed to a quotient such that the two varifold definitions agree.
\ While this reduction is always possible for smoothly immersed Hamiltonian
stationary Lagrangian submanifolds, we cannot rule out that the limiting
object may have different multiplicities at different points. Theorem 1.1 is
to be interpreted with this in mind, see Definition \ref{HsLV}.

It is illustrative to consider the 1-dimensional case. It is known that a
smooth curve in $\mathbb{C}$ is always Lagrangian and its Lagrangian phase
function is harmonic if and only if it is part of a straight line or a circle.
We address the codimension condition in our extension result and regularity on
the immersion for the compactness from a viewpoint based on the first
variation of 1-varifolds:

\begin{enumerate}
\item If $\iota:M^{1}\rightarrow\mathbb{C}$ is a Hamiltonian stationary
immersion where $M^{1}$ is compact, then its image $\iota(M^{1})$ is a circle.
The radius is uniformly bounded above from the length bound, although not
below. Thus, it is easy to see that Theorem \ref{conv_thm} is true for $n=1$.

\item Let $\gamma$ be the union of the rays $\gamma_{i}=t\vec{\eta_{i}}$ where
$\eta_{i}$ are unit linearly independent vectors in $\mathbb{R}^{2}%
,i=1,...,\ell$. Assign a multiplicity $m_{i}\in\mathbb{N}$ to $\gamma_{i}$. So
$(\gamma,\mu)$ is a 1-varifold for the measure $\mu=\sum_{i=1}^{\ell}%
m_{i}\,\mbox{d}\mu_{i}$ where $\mbox{d}\mu_{i}$ is the Euclidean length
element of $\gamma_{i}$. The first variation of $\gamma$ is given by
\[
\delta\gamma(J\hbox{D}f)=\langle J\hbox{D}f(0,0),\sum_{i=1}^{\ell}m_{i}%
\vec{\eta_{i}}\rangle
\]
for any $f\in C_{0}^{\infty}(\mathbb{R}^{2})$. In general the first variation
cannot be zero for arbitrary $f$, hence $\gamma$ is not Hamiltonian
stationary, unless a balancing condition $\sum_{i=1}^{\ell}m_{i}\vec{\eta_{i}%
}=0$ is prescribed.

Two points to make: first, the codimension requirement in Theorem
\ref{ext_thm} cannot be removed, i.e. we cannot expect to extend solutions
across a codimension one set in general. Second, a generic polygonal curve is
Hamiltonian stationary in $\mathbb{C}\backslash\{\mbox{vertices}\}$ but not
Hamiltonian stationary in $\mathbb{C}$, enhancing our first point; and the
lack of $C^{1,1}$ control on the potential function shows that the immersions
in Theorem \ref{conv_thm} need to be at least $C^{1}$ in order to appeal to
the regularity theory in \cite{ChenWarren}, which plays a crucial role in the
current paper.

\item While the map $\iota:\mathbb{S}^{1}\rightarrow\mathbb{C}$ given by
$\iota(z)=z$ provides an obvious immersion, there are many more: any map
$z\rightarrow z^{m}$, for $m$ a positive integer gives a Hamiltonian
stationary immersion, with the varifold a multiple of the varifold defined by
$\iota.$ \ We show that while all such maps define Hamiltonian stationary
varifolds (Proposition \ref{L=M}), the varifolds can be represented by a
\textquotedblleft canonical" immersion; see Proposition \ref{canonical}.

\item Consider $M^{1}=\mathbb{S}^{1}\sqcup\mathbb{S}^{1}$ and a sequence of
immersions $\iota_{k}:M^{1}$ $\rightarrow\mathbb{C}$ such that the image is a
pair of concentric circles with radii $1$ and $1+\frac{1}{k}.$ \ The limiting
object in the varifold topology will be a double copy of the unit circle, in
particular will have measure $2\mathscr{H}^{1}.$ \ If we \textquotedblleft
forget" the limiting immersion and consider only the point set, the object
will not be the limit in the varifold topology. \ While we avoid discussing
disconnected source manifolds in the current paper, this example suggests the
weighted definition of varifolds is more flexible in a broad setting. \ 
\end{enumerate}

\vspace{0.2cm}

We now outline the structure of the paper:

\vspace{.2cm}

In section 2, we set up basic framework for dealing with properly immersed
Hamiltonian stationary Lagrangian submanifolds. In particular, for a proper
Lagrangian immersion $\iota:M\rightarrow\mathbb{C}^{n}$, we show equivalence
of $L=\iota(M)$ being Hamiltonian stationary (seemingly weaker due to
non-embedded points) and the local embedding being Hamiltonian stationary.
This leads to the definition of Hamiltonian stationary varifolds which fits
naturally in convergence of a sequence of immersed ones. An immersed
Hamiltonian stationary Lagrangian submanifold defines a varifold in a natural
way, and these objects are compact in the space of varifolds. In later
sections we show this compactness is strong enough to retain the Hamiltonian
stationary Lagrangian property. We show in Proposition \ref{canonical} that if
a point set $L$ is the image of \emph{any} differentiable Hamiltonian
stationary immersion from a connected closed manifold $M,$ then there is a
canonical choice of manifold $\tilde M$ and immersion $\tilde{\iota}:\tilde M
\rightarrow\mathbb{C}^{n}$ such that $\tilde{\iota}(\tilde{M})=L$ with a
varifold structure such that the measure associated to $L$ will generically be
the $n$-dimensional Hausdorff measure $\mathscr{H}^{n}$. This structure result
relies heavily on analytic continuation arguments in Proposition \ref{L=M}
that follow from results in \cite{ChenWarren}.

In section 3, we derive curvature and smoothness estimates. For the
Hamiltonian stationary system we must work around the lack of some important
tools available in the minimal submanifold setting. While an approximate
monotonicity formula (for contact stationary surfaces) has been shown
(\cite[Section 3]{SWJDG}), this formula is considerably more complicated to
derive and less potent to apply than the corresponding formula often used in
the minimal surface case, which we clearly do not have. Simons' identity
(\cite{Simons}) plays an important role for minimal submanifolds in deriving
higher order estimates in terms of the second fundamental form $A$ and in
proving the $\varepsilon$-regularity (cf. \cite{CSc}, \cite{An}). However,
such a useful technique is not available for the Hamiltonian stationary case;
terms arising from $\nabla^{2}H$ in $\Delta_{g}|A|^{2}$ are not reduced to
lower order terms of $A$. Instead, we use a priori estimates for the potential
function $u$ by viewing \eqref{hseq} as a second order elliptic operator
$\Delta_{g}$ acting on the fully nonlinear second order elliptic operator
$\theta$ as in \cite{ChenWarren}. All this relies on, in an essential way,
writing $\theta$ as the summation of the arctangents as in \eqref{theta}. In a
general Calabi-Yau manifold $(M,\omega,J,\Omega)$ other than ${\mathbb{C}}%
^{n}$, the Lagrangian phase $\theta$ need not admit such an expression even as
a leading term, when writing the Lagrangian submanifold locally as a gradient
graph over its (Lagrangian) tangent space in the Darboux coordinates. The real
part of the nowhere vanishing holomorphic $n$-form $\Omega$ that defines
$\theta$ as a calibrating $n$-form does not necessarily take a simple form in
the Darboux coordinate system.

In section 4, we show in Theorem \ref{ext_thm} that a Hamiltonian stationary
Lagrangian submanifold away from a small set with Hausdorff codimension at
least 2, but locally non-collapsing in volume according to its Hausdorff
dimension, extends across the set as a Hamiltonian stationary varifold
provided its mean curvature $H$ is in $L^{n}$ and a volume condition near the
small set is satisfied. This volume condition follows directly from the
monotonicity formula if $H=0$, and it is also valid if the set is of isolated
points and $n\geq2$, see Proposition \ref{points}.

In section 5, we prove Theorem \ref{conv_thm}. The structure of the
convergence part in the proof is similar to that in \cite{CSc} and \cite{An}.
To show the limit is Hamiltonian stationary, we invoke our extension result
Theorem \ref{ext_thm}.


\section{Hamiltonian stationary immersions}

In this section we set up the basic framework for dealing with compact smooth
Lagrangian Hamiltonian stationary immersions.

We will need to deal with immersed submanifolds that may be non-embedded, so
we define the following.

\begin{definition}
\label{emb-comp} \emph{ Let $L$ be an immersed submanifold, given by
$\iota:M^{n}\rightarrow{\mathbb{R}}^{2n}$. Given any connected open set
$U\subset{\mathbb{R}}^{2n}$, decompose the inverse image into connected
components as
\[
\iota^{-1}(U)=\bigsqcup_{i}E_{i}.
\]
If $\iota$ restricted to each }$E_{i}$\emph{ is a smooth embedding into
$\mathbb{R}^{2n}$, then we say that each }%
\[
\Sigma_{i}={\iota\left(  E_{i}\right)  }%
\]
\emph{ is an \textit{embedded connected component }of $U\cap L$ and that
}$\iota$\emph{ \textit{splits into embedded components on $U$.} }
\end{definition}

\begin{proposition}
\label{locallyembedded} Let $\iota:M\rightarrow{\mathbb{R}}^{2n}$ be a proper
immersion of a smooth manifold $M$, and set $L=\iota(M)$. For any point $y\in
L$, there is an open ball $B_{r}^{2n}(y)$ such that $\iota$ splits into
embedded components on $B_{r}^{2n}(y)$, and each component $\Sigma_{i}$
contains $y$.
\end{proposition}

\begin{proof}
For any fixed $y\in\iota(M)$, since $\iota$ is a proper immersion, the
pre-image of $y$ is a finite set $\iota^{-1}(\{y\})=\{x_{1},...,x_{m}\}$. Let
$B(x_{1}),...,B(x_{m})$ be disjoint coordinate balls (with respect to
arbitrary charts for $M)$ centered at $x_{i},$ with $\iota$ injective on each
$B(x_{i})$. Take a decreasing sequence $r_{k}\rightarrow0$. Let
\[
S_{k}=\iota^{-1}({B_{r_{k}}^{2n}(y)})\bigcap\left(  M\backslash\bigcup
_{i=1}^{m}B(x_{i})\right)  .
\]
Clearly, $S_{k+1}\subset S_{k}$. If there exists $x$ in all $S_{k}$ then
$\iota(x)=y$. So $x\in\{x_{1},...,x_{m}\}$, but this violates the definition
of $S_{k}$. Thus there is some $k_{0}$ such that $S_{k_{0}}=\emptyset$. Then
\[
\iota^{-1}(B_{r_{k_{0}}}^{2n}(y))\subseteq\bigcup_{i=1}^{m}B(x_{i})
\]
and this implies
\[
\iota(M)\cap B_{r_{k_{0}}}^{2n}(y)\subset\bigcup_{i=1}^{m}\iota(B(x_{i}))
\]
and then
\[
\iota(M)\cap B_{r_{k_{0}}}^{2n}(y)=B_{r_{k_{0}}}^{2n}(y)\bigcap\bigcup
_{i=1}^{m}\iota(B(x_{i}))=\bigcup_{i=1}^{m}\iota(B(x_{i}))\cap B_{r_{k_{0}}%
}^{2n}(y).
\]

We finish the proof by showing that $\iota(B(x_{i}))\cap B_{r}^{2n}(y)$ is
connected for all $r\geq r_{i}$ for some positive $r_{i}$ and then taking the
smallest $r_{i},i=1,...,m$. Represent $\iota(B(x_{i}))$ locally as a graph of
a vector valued function $F:B_{\rho}^{n}(0)\subset{\mathbb{R}}^{n}%
\rightarrow\mathbb{R}^{n}$, where we identify $T_{y}\iota(B(x_{i}))$ with
${\mathbb{R}}^{n}$ and $y$ with $0$; we further assume
$F(0)=0,\hbox{D}F(0)=0,|\mbox{D}F|\leq C(\rho)$ on $B_{\rho}^{n}(0)$. Then any
point $x$ with $(x,F(x))\in\partial B_{\rho}^{2n}(y)$ satisfies
\[
\rho^{2}=|x|^{2}+|F(x)|^{2}\leq(1+C_{\rho}^{2})|x|^{2}%
\]
therefore
\[
|x|\geq\frac{\rho}{\sqrt{1+C_{\rho}^{2}}}.
\]
If $\iota(B(x_{i}))\cap B_{\rho}^{2n}(y)$ is disconnected, there must be a
point $p\in\iota(B(x_{i}))\cap\partial B_{\rho}^{2n}(y)$ that is not on the
connected component containing $y$. On the ray $\sigma(t)=tx_{p}/|x_{p}|$ from
$0$ to $x_{p}$ in $B_{\rho}^{n}(0)$ where $p=(x_{p},F(x_{p}))$, there must be
two distinct points $\sigma(t_{1}),\sigma(t_{2})$ with
\[
\frac{\rho}{\sqrt{1+C_{\rho}^{2}}}\leq t_{1},t_{2}\leq\rho
\]
such that $t_{1}$ is the last departing time for $\iota(B(x_{i}))$ to leave
$B_{\rho}^{2n}(y)$ and $t_{2}$ is the first returning time. Thus, for the
smooth function
\[
f(t)=|x(t)|^{2}+|F(x(t))|^{2}%
\]
we have $f^{\prime}(t_{1})\geq0$ and $f^{\prime}(t_{2})\leq0$. So there is
$t_{0}\in\lbrack t_{1},t_{2}]$ with $f^{\prime}(t_{0})=0$, i.e.
\[
x(t_{0})\cdot\sigma^{\prime}(t_{0})+F(x(t_{0}))\cdot\mbox{D}F_{x(t_{0}%
)}(\sigma^{\prime}(t_{0}))=0.
\]
Since $\sigma^{\prime}(t_{0})$ is a unit vector, we have
\[
\frac{\rho}{\sqrt{1+C_{\rho}^{2}}}\leq|x(t_{0})|=|x(t_{0})\cdot\sigma^{\prime
}(t_{0})|\leq C_{\rho}|F(x(t_{0}))|\leq C_{\rho}\rho.
\]
But this becomes impossible for small $\rho$ since $C_{\rho}=\left\Vert
DF\right\Vert _{L^{\infty}(B_{\rho}(0))}\rightarrow0$, and we have a
contradiction. We conclude $\iota(B(x_{i}))\cap B_{\rho}^{2n}(y)$ must be connected.
\end{proof}

\begin{definition}
\emph{ \label{HsLV}Let $V$ be an integral rectifiable $k$-varifold on an open
subset $U$ of $\mathbb{C}^{n}$ with generalized mean curvature $\mathcal{H}$.
We say $V$ is \textit{Hamiltonian stationary in $U$} if
\begin{equation}
\int_{U}\langle J\,\hbox{D}f,\mathcal{H}\rangle\hbox{d}\mu_{V}%
=0\label{varform}%
\end{equation}
for any $f\in C_{0}^{\infty}(U)$. If $k=n$ and every approximate tangent space
$T_{x}V$ is a Lagrangian $n$-plane in $\mathbb{C}^{n}$, we say $V$ is a
\textit{Lagrangian varifold}. If a Lagrangian varifold is Hamiltonian
stationary, it is a \textit{Hamiltonian stationary Lagrangian $n$-varifold. }}
\end{definition}

\begin{definition}
\label{piHsL} \emph{Let $L$ be Hamiltonian stationary Lagrangian $n$-varifold
$L$ that is defined by a proper immersion $M^{n}\rightarrow\mathbb{C}^{n}$. We
say that} $L$ is a properly immersed Hamiltonian stationary Lagrangian
submanifold in $\mathbb{C}^{n}.$
\end{definition}

\begin{proposition}
\label{L=M} Let $\iota:M^{n}\rightarrow\mathbb{C}^{n}$ define a properly
immersed Hamiltonian stationary Lagrangian submanifold $L$ in $\mathbb{C}^{n}$
with $M$ connected. Then

\begin{enumerate}
\item The Lagrangian phase function $\theta$ of each embedded connected
component of $L\cap U$ satisfies (\ref{hseq}) for any open subset $U$ of
$\mathbb{C}^{n}$. Conversely, if \eqref{hseq} holds on each embedded connected
component then $L$ is Hamiltonian stationary.

\item (Unique continuation) If the intersection of two embedded connected
components contains an open set, then they coincide.
\end{enumerate}
\end{proposition}

\begin{proof}
For any $y_{0}\in L,$ let $L\cap B_{r}^{2n}(y_{0})$ decompose into embedded
connected components $\Sigma_{1},\dots,\Sigma_{m}$ as in Proposition
\ref{locallyembedded}. Each $\Sigma_{i}$ is Lagrangian with a Lagrangian angle
$\theta_{i}:\Sigma_{i}\rightarrow\mathbb{R}/2\pi\mathbb{Z}$ defined (up to
orientation) by
\[
dz^{1}\wedge\cdots\wedge dz^{n}|_{\Sigma_{i}}=e^{\sqrt{-1}\theta_{i}}%
d\mu_{g_{i}}%
\]
and its mean curvature vector satisfies $H_{i}=J\nabla\theta_{i}$ where $J$ is
the complex structure on ${\mathbb{C}}^{n}$ (cf. \cite{HL}) and $d\mu_{g_{i}}$
is the volume form of the induced metric $g_{i}$ on $\Sigma_{i}$ by the
Euclidean metric on $\mathbb{R}^{2n}$. \ We divide the point set $L$ into two
pieces. We say a point $y\in L$ is an embedded point if there is an open set
$W$ in ${\mathbb{R}}^{2n}$ containing $y$ so that the point set $L\cap W$ is
an embedded submanifold in $\mathbb{R}^{2n}$ and let $\mathcal{E}$ be the set
of all embedded points of $L$. We show first that \eqref{hseq} holds on
$\mathcal{E}$, and then argue that for each $\Sigma_{i},$ $\mathcal{E\cap}$
$\Sigma_{i}$ is dense in $\Sigma_{i}$.

For any $y\in\mathcal{E}$, $L\cap$ $B_{r}^{2n}(y)$ is an embedded submanifold
for some $r>0,$ and by Proposition \ref{locallyembedded}, there exists a
sufficiently small ball $B_{r_{0}}^{2n}(y)$ in $\mathbb{R}^{2n}$, such that
$\iota^{-1}(B_{r_{0}}^{2n}(y))$ is a finite disjoint union of $E_{1}%
,...,E_{m(y)}$, and $\iota|_{E_{i}}$ is an embedding with
\begin{equation}
\Sigma_{i}:=\iota(E_{i})=L\cap B_{r_{0}}^{2n}(y) \label{all_the_same}%
\end{equation}
for each $i$, and $m(y)$ is constant on $L\cap B_{r_{0}}^{2n}(y)$. Pulling
back the Euclidean metric on $\mathbb{R}^{2n}$ and the $n$-form $dz^{1}%
\wedge\cdots\wedge dz^{n}$ by $\iota$, we see that $E_{i},E_{j}$ are isometric
in their induced metrics, and the Lagrangian angles $\theta_{\iota}$ are the
same, since $\iota|_{E_{j}}^{-1}\circ\iota|_{E_{i}}:E_{i}\rightarrow E_{j}$ is
a diffeomorphism.

Now for any $\phi\in C_{c}^{\infty}(M)$ with support in $E_{1}$, define
\[
\varphi(y)=%
\begin{cases}
\phi(\iota^{-1}(\{y\})) & \mbox{if} \,\,\,y\in L\cap B_{r_{0}}^{2n}(y)\\
0 & \mbox{if}\,\,\,\text{otherwise}%
\end{cases}
\]
which is a well-defined function and smooth on $L$ and can be extended to a
function $f\in C_{c}^{\infty}(\mathbb{R}^{2n})$. Since $L$ is Hamiltonian
stationary, by (\ref{varform}), we have
\begin{align*}
0  &  =\int_{L}\langle J\,\mbox{D}f(y),\mathcal{H}(y)\rangle\hbox{d}\mu
_{L}(y)\\
&  =\int_{M}\langle J\,\mbox{D}f(\iota(x)),H_{\iota}(x)\rangle\,\hbox{d}\mu
_{M}\\
&  =\int_{M}\langle J\,\mbox{D}f(\iota(x)),J\nabla\left(  \theta\circ
\iota\right)  (x)\rangle\hbox{d}\mu_{M}\\
&  =-\int_{M}f(\iota(x))\Delta\left(  \theta\circ\iota\right)  (x)\hbox{d}\mu
_{M}\\
&  =-\int_{E_{1}\cup\cdots\cup E_{m(y)}}f(\iota(x))\Delta\left(  \theta
\circ\iota\right)  (x)\hbox{d}\mu_{M}\\
&  =-m(y)\int_{E_{1}}\phi\,\Delta\theta_{\iota}\hbox{d}\mu_{M}%
\end{align*}
and the harmonicity of $\left(  \theta\circ\iota\right)  $ on $E_{1}$ follows
as $\phi$ is arbitrary function in $C_{c}^{\infty}(E_{1})$. By
(\ref{all_the_same}), \eqref{hseq} holds on $\mathcal{E}\cap\Sigma_{i}$.

Next, we show that $\mathcal{E}\cap\Sigma_{i}$ is dense in $\Sigma_{i}$.
First, we consider two embedded connected components $\Sigma_{i}$ and
$\Sigma_{j}$ (if there is only one, we are done), and let $\mathcal{E}_{ij}$
be the set
\[
\mathcal{E}_{ij}=\left\{  y\in\Sigma_{i}:\left(  \Sigma_{i}\cup\Sigma
_{j}\right)  \cap B_{r}^{2n}(y)\text{ is embedded for some }r>0\text{
}\right\}
\]
The set $\mathcal{E}_{ij}$ is open in $\Sigma_{i}$. The complement in
$\Sigma_{i}$
\[
\mathcal{E}_{ij}^{c}=\Sigma_{i}\backslash\mathcal{E}_{ij}\subseteq\Sigma
_{i}\cap\Sigma_{j}%
\]
has no interior points in $\Sigma_{i}$ : If
\[
B_{r}^{2n}(y)\cap\Sigma_{i}\subseteq\Sigma_{i}\cap\Sigma_{j}%
\]
then necessarily
\[
B_{r}^{2n}(y)\cap\left(  \Sigma_{i}\cap\Sigma_{j}\right)  =B_{r}^{2n}%
(y)\cap\Sigma_{i}%
\]
so $\Sigma_{i}\cap\Sigma_{j}$ is embedded near $y$, and $y\in\mathcal{E}%
_{ij}.$ Thus $\mathcal{E}_{ij}^{c}$ is closed and nowhere dense in $\Sigma
_{i}$, in turn, $\mathcal{E}_{ij}$ is dense and open in $\Sigma_{i}$. Now we
claim $\mathcal{E}\cap\Sigma_{i}=\Sigma_{i}\backslash\cup_{j}\mathcal{E}%
_{ij}^{c}$. \ To see this, if $y\in\Sigma_{i}\backslash\cup_{j}\mathcal{E}%
_{ij}^{c}$, then $y\in\Sigma_{i}\cap_{j}E_{ij}$. For each $j$ there is a
neighborhood $U_{ij}$ of $y$ so that $\Sigma_{i}\cap U_{ij}$ is an embedded
submanifold, then $\Sigma_{i}\cap_{j}U_{ij}$ is embedded since there are only
finitely many $j$, thus $y\in\mathcal{E}\cap\Sigma_{i}$. The other direction
is obvious. Combining the above, we see that \eqref{hseq} holds on the dense
set $\mathcal{E}\cap\Sigma_{i}$. \ Because $\Delta\theta_{\iota}$ is a smooth
function on $\Sigma_{i}$ we conclude that \eqref{hseq} holds on $\Sigma_{i}.$

Next, to show the converse, let $\{B_{\alpha}\}$ be a countable collection of
open balls in $\mathbb{R}^{2n}$ such that $\{B_{\alpha}\}$ covers $L$
and\ $\iota$ splits into embedded connected components on each $B_{\alpha}$.
\ \ Denote the components $\Sigma_{i,\alpha},$ that is, \ let $\iota\left(
E_{i,\alpha}\right)  =\Sigma_{i,\alpha}$ and let $g$ be the metric on
$E_{i,\alpha}$ such that $\iota$ is an isometry from $E_{i,\alpha}$ to the
induced metric on $\Sigma_{i,\alpha},$ which we denote $g_{i,\alpha}.$ Let
$\{\varphi_{\alpha}\}$ be a partition of unity subordinate to the open cover
$\{B_{\alpha}\}$ of the open set $\cup_{\alpha}B_{\alpha}$. For any $h\in
C_{c}^{\infty}(\mathbb{R}^{2n})$,
\begin{align*}
\int_{U}\langle J\hbox{D}h,H\rangle d\mu_{L} &  =\sum_{\alpha}\int_{U}\langle
J\hbox{D}(\varphi_{\alpha}h),H\rangle d\mu_{L}\\
&  =\sum_{\alpha}\int_{B_{\alpha}\cap L}\langle J\hbox{D}(\varphi_{\alpha
}h),H\rangle d\mu_{L}\\
&  =\sum_{\alpha}\sum_{i}\int_{\Sigma_{i,\alpha}}\langle J\hbox{D}(\varphi
_{\alpha}h),H_{\Sigma_{i,\alpha}}\rangle d\mu_{g_{i,\alpha}}\\
&  =\sum_{\alpha}\sum_{i}\int_{E_{i,\alpha}}\langle\nabla_{g}(\varphi_{\alpha
}h)\circ\iota,\nabla_{g}\left(  \theta\circ\iota\right)  \rangle d\mu_{g}.
\end{align*}
Therefore $L$ is Hamiltonian stationary as $\theta\circ\iota$ is harmonic on
each $E_{i,\alpha},$ and the restriction of $\varphi_{\alpha}\in C_{c}%
^{\infty}(\Sigma_{i,\alpha})$. \ This proves (1). \ 

Suppose that $\Sigma_{1}\cap\Sigma_{2}$ contains a nonempty connected set
$W_{0}$ that is open with respect to the topology on both $\Sigma_{1}$ and
$\Sigma_{2}.$ Let $W$ be the union of all connected subsets of $\Sigma_{1}%
\cap\Sigma_{2}$ that are open in both $\Sigma_{1}$ and $\Sigma_{2}$ and that
contain $W_{0}$. \ We claim that $W=\Sigma_{1}=$ $\Sigma_{2}.$ \ Let $\partial
W:=\overline{W}\backslash W\not =\emptyset$, and consider two cases.

Case 1: $\partial W\cap\Sigma_{1}=\emptyset$ (or $\partial W\cap\Sigma
_{2}=\emptyset\,$) . First $\partial W\subseteq\partial\Sigma_{1}%
\subset\partial U$, as the immersion is proper. If $\Sigma_{1}\backslash
W\not =\emptyset$, let $p\in\Sigma_{1}\backslash W$ be an arbitrary point. If
every neighborhood of $p$ in $\Sigma_{1}$ intersects $W$ then $p\in\partial
W$, in turn $p\in\Sigma_{1}\cap\partial\Sigma_{1}=\emptyset$, as $\Sigma_{1}$
is embedded. \ So there is a neighborhood of $p$ in $\Sigma_{1}$ not
intersecting $W$, and we conclude that $\Sigma_{1}\backslash W$ is open in
$\Sigma_{1}$. But this is impossible since $\Sigma_{1}=W\cup(\Sigma
_{1}\backslash W)$ is connected. This contradicts $p\in\Sigma_{1}\backslash
W,$ so we conclude that $\Sigma_{1}\subset W\subset\Sigma_{1}\cap\Sigma_{2}.$
\ In particular, $\Sigma_{1}=W.$ \ Now if $\partial W\cap\Sigma_{2}$
$=\emptyset,$ repeat the argument to conclude that $\Sigma_{2}=W.$ \ If
$\partial W\cap\Sigma_{2}$ $\neq\emptyset,$ then $\partial\Sigma_{1}\cap
\Sigma_{2}\neq\emptyset,$ which leads to a contradiction, as then
$\partial\Sigma_{1}\subset\partial U$ and $\Sigma_{2}\cap\partial
U=\emptyset.$ \ Thus $\Sigma_{1}=W=\Sigma_{2}.$

Case 2: $\partial W\cap\Sigma_{1}\not =\emptyset$ and $\partial W\cap
\Sigma_{2}\not =\emptyset$. Let $q\in\partial W\cap\Sigma_{1}$. There is a
sequence $q_{k}\in W\rightarrow q$. As $W\subseteq\Sigma_{2}$ and $\Sigma
_{1}\cap\partial\Sigma_{2}=\emptyset$ we have $q\in\Sigma_{2}$, so $q\in
\Sigma_{1}\cap\Sigma_{2}\cap\partial W$. Write $\Sigma_{1},\Sigma_{2}$ locally
over their common tangent space $T_{q}$ at $q$ as graphs of $Du_{1},Du_{2}$
for some smooth functions $u_{1},u_{2}$ on some ball $B_{r}(q)$ in $T_{q}$,
with $u_{1}(q)=u_{2}(q)$ and $Du_{1}(q)=Du_{2}(q)=0$. By maximality of $W$,
$\Sigma_{1},\Sigma_{2}$ coincide over $B_{r}(q)\cap W$ and are distinct on
$B_{r}(q)\backslash W$ for all small $r$. Thus we can arrange $u_{1}=u_{2}$ on
$B_{r}(q)\cap W$ and $Du_{1}(x_{k})\not =Du_{2}(x_{k})$ for a sequence of
$x_{k}\in B_{r}(q)\backslash W\rightarrow q$. However, as solutions to
\eqref{hseq}, both $u_{1},u_{2}$ are analytic, by \cite[p. 203]{Morrey58}.
Thus, $u_{1}=u_{2}$ on $B_{r}(q)$ as they agree on $B_{r}(q)\cap W$. We have a
contradiction, leaving us with the conclusion in Case 1, that $\Sigma
_{1}=W=\Sigma_{2}$
\end{proof}

\begin{definition}
\emph{A proper immersion $\iota:M\rightarrow\mathbb{C}^{n}$ is} reduced
\emph{if there is an open dense subset of points from $M$ on which $\iota$ is
injective. }
\end{definition}

Equivalently, the proper immersion $\iota$ is reduced if the varifold
structure defined on the point set $\iota(M)$ with the Hausdorff measure
agrees with the varifold structure defined by pushing forward the induced
volume measure.

\begin{proposition}
\label{canonical} Suppose that $L$ is a compact immersed Hamiltonian
stationary Lagrangian submanifold in $\mathbb{C}^{n}.$ There is a smooth
manifold $\tilde{M}^{n},$ unique up to diffeomorphism, such that $\tilde
{\iota}:$ $\tilde{M}^{n}\rightarrow\mathbb{C}^{n}$ defines $L$ and is reduced.
\end{proposition}

\begin{proof}
Let $\iota:M\rightarrow\mathbb{C}^{n}$ define $L.$ \ Define%
\[
m:M\rightarrow%
\mathbb{N}
\]
by
\[
m(x)=\#\left\{  \iota^{-1}(\iota(x))\right\}  .
\]
In general, the function $y\mapsto$ $\#\left\{  \iota^{-1}(y)\right\}  $ is
upper semicontinuous on $\mathbb{C}^{n}$ whenever $\iota$ is a proper
immersion. \ It follows that $m$ is upper semicontinuous on $M.$

Let%
\[
m_{1}=\min_{x\in M}m(x)
\]
and
\begin{align}
O_{1}  &  =\left\{  x\in M:m(x)=m_{1}\right\} \label{O1}\\
&  =\left\{  x\in M:m(x)<m_{1}+\frac{1}{2}\right\}  .\nonumber
\end{align}
By upper semicontinuity, we see that $O_{1}$ is open. $\ $Next we claim that
$O_{1}$ is dense. $\ $Suppose that $O_{1}^{c}$ has nontrivial interior $V$,
and let $x_{1}\in\partial V.$ \ As a boundary point, every neighborhood of
$x_{1}$ intersects both $\left\{  m(x)>m_{1}\right\}  $ and $\left\{
m(x)=m_{1}\right\}  $ and by upper semicontinuity, we have that that
$m(x_{1})>m_{1}.$ By Proposition \ref{locallyembedded} there is a neighborhood
$U$ of $\iota(x_{1})$ that splits into exactly $m(x_{1})$ embedded connected
components; let $E_{1}\subset M$ be the one containing $x_{1}$ and label the
others $E_{2},...,E_{m(x_{1})}.$ \ Now every open set containing $x_{1}$
intersects $V$, thus $V\cap E_{1}$ is an non-empty open subset of $E_{1}$ on
which $m\geq m_{1}+1.$ In particular, for all $x\in V\cap E_{1}$ there is some
subset $\alpha\left(  x\right)  \subset\left\{  2,...,m(x_{0})\right\}  $ with
$\left\vert \alpha\left(  x\right)  \right\vert =m_{1}$ such that
\[
\iota(x)\in%
{\displaystyle\bigcap_{j\in\alpha\left(  x\right)  }}
\iota(E_{j})
\]
thus
\[
V\cap E_{1}\subset%
{\displaystyle\bigcup\limits_{\left\vert \alpha\right\vert =m_{1}}}
{\displaystyle\bigcap_{j\in\alpha}}
\iota_{|_{E_{1}}}^{-1}\left(  \iota(E_{j})\right)
\]
which is a finite union of closed sets. Applying the Baire category theorem,
we conclude that there is a set of at least $m_{1}$ components that intersect
not only each other but also $\iota(E_{1})$ in an open set. By Proposition
\ref{L=M} (2), we conclude these components must coincide on all of $U,$ in
particular, we have $m(x)\geq m_{1}+1$ in a neighborhood of $x_{1}$,
contradicting our assumption that $x_{1}$ was a boundary point of $V$. It
follows that $O_{1}^{c}$ has empty interior, so $O_{1}$ is dense and
$m(x)=m_{1}$ on an open dense set.

Next, we define a quotient map $\pi:M\rightarrow\tilde{M}:=M/\!\sim$ as
follows. For $x,y\in M$, declare $x\sim y$ if both

\begin{enumerate}
\item $\iota(x)=\iota(y)$;

\item There exists a neighborhood $U_{x}$ of $x$ and a neighborhood $U_{y}$ of
$y$ in $M$ such that
\[
\iota(U_{x})=\iota(U_{y})
\]
and both are embedded connected components.
\end{enumerate}

Clearly, $\sim$ is reflexive, symmetric, and transitive (using intersections
of open sets), so defines a quotient map, and there is a unique quotient
topology on $\tilde{M}$ . By definition, the neighborhoods $\left\{
U_{x^{\prime}}:x^{\prime}\in\lbrack x]\right\}  $ provide an even covering of
a neighborhood of $\left[  x\right]  $ thus $\pi$ is a topological covering
map. It follows (\cite[Proposition 4.40]{Lee}) that $\tilde{M}$ has a unique
smooth manifold structure such that $\pi$ is a smooth covering map. By
condition 1), $\iota$ agrees on fibers of $\pi$, thus there is a unique map
$\tilde{\iota}$ $:\tilde{M}\rightarrow\mathbb{C}^{n}$ (\cite[Theorem
4.30]{Lee}) such that $\tilde{\iota}\circ\pi=\iota$. Now $\tilde{\iota}$ is an
immersion (this can be verified locally on an evenly covered neighorhoods)
from a compact manifold.

At any point $x\in O_{1}$, consider $[x_{1}],[x_{2}]\in\tilde{\iota}%
^{-1}(\iota(x))$, so $\iota(x_{1})=\iota(x_{2})=\iota(x)$. As $\iota$ is
immersive, there exist neighborhhoods $U_{x_{1}},U_{x_{2}}$ of $x_{1},x_{2}$
in $M$ respectively such that the restriction of $\iota$ on each neighborhood
is diffeomorphic onto its image and is an embedding into $\mathbb{R}^{2n}$.
Furthermore, we may assume $\iota(U_{x_{1}})=\iota(U_{x_{2}})$ by taking
$\iota|_{U_{x_{i}}}^{-1}(\iota(U_{x_{1}})\cap\iota(U_{x_{2}}))$ as $U_{x_{i}},
i=1,2$. So $[x_{1}]=[x_{2}]$, and $\tilde\iota$ is injective on the open dense
set $\tilde{\iota}^{-1}(\iota(O_{1}))$ in $\tilde{M}$. Thus, $\tilde{\iota}$
is reduced.


Finally, we argue that the smooth structure and topology are unique. Let
$\iota:M\rightarrow\mathbb{C}^{n}$ be any reduced Hamiltonian stationary
immersion defining $L$. Take an open cover (with respect to $\mathbb{C}^{n}%
$-topology$)$ of balls around points in $L$ on which $L$ splits into embedded
connected components, and choose a finite cover, say $N$ of such balls,
$B_{r_{1}}(y_{1})...,...B_{r_{N}}(y_{N})$. Define $E_{j,k}\subset M$ by
\[
\iota^{-1}(B_{r_{j}}(y_{j}))=\bigcup_{k=1}^{m_{j}}E_{j,k}%
\]
where $m_{j}$ is the number of connected components associated to $B_{r_{j}%
}(y_{j}).$ Now, let $\iota^{\prime}:M^{\prime}\rightarrow\mathbb{C}^{n}$ be
another reduced immersion, which determines the same set $L$. We can choose
the same set of balls $B_{r_{i}}(y_{i})$ in the same order, and define the
$E_{j,k}^{\prime}\subset M^{\prime}$ with the same choice of decomposition
into embedded connected components, noting that the decomposition is
determined by Proposition \ref{L=M} (2) together with the fact we have chosen
the immersion to be reduced: Each component is unique, so there can be no
discrepancy. Consider the map%
\[
F:M\rightarrow M^{\prime}%
\]
defined by
\[
F(x):=\iota^{\prime-1}|_{\Sigma_{j,k}}\circ\iota(x)\text{ whenever }x\in
E_{j,k}.
\]
Now each $x$ is contained in at least one $E_{j,k}$, as $\iota$ and
$\iota^{\prime-1}|_{\Sigma_{j,k}}$ are smooth, the map is clearly smooth,
provided it is not multiply defined, so we must show that it is well-defined.
Suppose that $x\in E_{j,k}$ $\cap E_{j^{\ast},k^{\ast}}$. If $\iota(x)$ is
contained in a set $U$ where $\iota^{\prime-1}$ is well-defined, uniqueness of
the definition is clear: $\iota$ and $\iota^{\prime-1}$ are both well-defined,
so $F(x)$ is defined regardless of which set we choose : $E_{j,k}$ or
$E_{j^{\ast},k^{\ast}}$. Now if $m(x)>1,$ we may use smooth continuation at
$x$, noting that $F(x)$ is well-defined at points in any neighborhood of $x.$
In particular, we have smooth maps%
\begin{align*}
\iota^{\prime-1}|_{E_{j,k}}\circ\iota(x)  &  :E_{j,k}\cap E_{j^{\ast},k^{\ast
}}\rightarrow M^{\prime}\\
\iota^{\prime-1}|_{E_{j^{\ast},k^{\ast}}}\circ\iota(x)  &  :E_{j,k}\cap
E_{j^{\ast},k^{\ast}}\rightarrow M^{\prime}%
\end{align*}
that are smooth individually and agree on an open dense set near $x$. They
must then agree completely on their common domain of definition. Thus $F(x)$
is well-defined. A smooth inverse can easily be constructed in the same way,
so we conclude that $F$ is a diffeomorphism.
\end{proof}

\section{ Curvature and higher order estimates}

\subsection{Graphical representation of Lagrangian submanifolds}

We begin with rephrasing, for Lagrangian submanifolds, a well known fact about
local graphical representation of embedded submanifolds, that gives a precise
lower bound, in terms of the length of the second fundamental form, on the
size of a ball in the tangent space over which the Lagrangian submanifold is a
graph of the gradient of a potential function with uniform Hessian bound. The
bounds are written in a convenient form for the rotation argument in the proof
of Proposition \ref{r1}.

\begin{lemma}
\label{lemma1} Let $L$ be a properly and smoothly immersed connected
Lagrangian submanifold in ${\mathbb{C}}^{n}$. Suppose that $\left\Vert
A\right\Vert _{\infty}\leq C$ and $\partial L\cap{B_{\rho_{0}}^{2n}%
(0)}=\emptyset$, where $A$ is the second fundamental form of $L$ and
$B_{\rho_{0}}^{2n}(0)$ is the ambient ball with radius $\rho_{0}(C)=\frac{\pi
}{12C}$ and $0\in L$. Then any embedded connected component $\Sigma$ of
$B_{\rho_{0}}^{{2n}}(0)\cap L$ containing $0$ is a gradient graph over a
region $\Omega\subset T_{0}\Sigma,$ that is, there is a function
$u:\Omega\rightarrow\mathbb{R}$ such that
\[
\Sigma=\left\{  (x,Du(x)),x\in\Omega\right\}
\]
and $\Omega$ contains the ball $B_{r_{0}}^{{n}}(0)\subset T_{0}L$, where
\begin{equation}
r_{0}(C)=\frac{\pi}{12C}\cos\frac{\pi}{12}\label{r_0}%
\end{equation}
Further,
\begin{equation}
\left\vert D^{2}u\right\vert \leq\tan\frac{\pi}{12}\text{ on }B_{r_{0}}%
^{n}(0).\label{HB}%
\end{equation}

\end{lemma}

\begin{proof}
Locally, any embedded Lagrangian submanifold $\Sigma$ is the gradient graph
over its tangent space $T_{0}\Sigma$ of a function $u$ with $\hbox{D}^{2}%
u(0)=0$, say over a ball $B_{\sigma_{0}}^{n}(0)$. Let $\lambda_{i}\left(
x\right)  $ be the eigenvalues of $D^{2}u(x)$. First we claim:
\[
\left\Vert \nabla_{g}\arctan\lambda_{i}\right\Vert \leq\left\Vert A\right\Vert
_{\infty}\leq C.
\]
To see this, consider the (3,0)-tensor
\[
A:T\Sigma\times T\Sigma\times N\Sigma\rightarrow\mathbb{R}%
\]
defined by
\[
A(X,Y,\vec{n})=\hbox{D}_{X}Y\cdot\vec{n}.
\]
with components
\[
A(\partial_{i},\partial_{j},n_{k})=u_{jki}%
\]
under a local coordinate frame $\partial_{1},...,\partial_{n}$, where
$n_{k}=J\partial_{k}\in N\Sigma$ as $\Sigma$ is Lagrangian. Thus
\begin{align*}
\left\Vert A\right\Vert ^{2} &  =\sum_{i,j,k}g^{ia}g^{jb}g^{kc}u_{ijk}%
u_{abc}\\
&  =\sum_{i,j,k}g^{ii}g^{jj}g^{kk}u_{ijk}^{2}%
\end{align*}
when $D^{2}u$ is diagonalized. For any $i,$
\[
\left\Vert \nabla_{g}\arctan\lambda_{i}\right\Vert ^{2}=g^{jj}\left(  \frac
{1}{1+\lambda_{i}^{2}}\right)  ^{2}u_{iij}^{2}%
\]
with $D^{2}u$ still diagonalized. Now recalling $g^{ii}=1/(1+\lambda_{i}^{2})$
when diagonalized,
\[
\sum_{i,j}g^{jj}g^{ii}g^{ii}u_{iij}^{2}\text{ \ }\leq\sum_{i,j}g^{jj}%
g^{ii}g^{ii}u_{iij}^{2}+\sum_{i,j,k\neq i}g^{jj}g^{ii}g^{kk}u_{ikj}^{2}%
=\sum_{i,j}g^{jj}g^{ii}g^{kk}u_{ikj}^{2}=\left\Vert A\right\Vert ^{2}.
\]
This proves the claim.

Next, let $v$ be any unit vector in $T_{0}\Sigma$ and let $\gamma
_{v}(s)=(sv,Du(sv))$ for $s\in\lbrack0,\sigma)$. Integrating along $\gamma
_{v}$ and using the claim, we see that the maximum value of $\arctan
\lambda_{i}$ satisfies
\[
\left\vert \arctan\lambda_{i}\right\vert \leq C\emph{l}(\gamma_{v})
\]
using $\emph{l}(\gamma_{v})$ to denote the length of the curve. \ Thus the
maximum slope of $\gamma_{v}$ (precisely, each planar curve $(sv,u_{i}%
(sv)),i=1,...,n$) satisfies
\begin{equation}
\left\vert \lambda_{i}\right\vert \leq\tan(C\emph{l}(\gamma_{v})).
\label{simple_bound}%
\end{equation}
Now since%
\[
\emph{l}(\gamma_{v})=\int_{0}^{\sigma}\sqrt{1+|DDu(sv)|^{2}}\,ds
\]
we have
\begin{align*}
\frac{d}{d\sigma}\emph{l}(\gamma_{v})  &  =\sqrt{1+|D^{2}u(\sigma v)|^{2}%
}\,\leq\sqrt{1+\tan^{2}(C\emph{l}(\gamma_{v}))}\\
&  =\sec\left(  C\emph{l}(\gamma_{v})\right)
\end{align*}
Integrating,
\[
\emph{l}(\gamma_{v})\leq\frac{1}{C}\arcsin(C\sigma)
\]
provided $\sigma\in(0,\frac{\pi}{2C}).$ \ So if we choose
\[
r_{0}=\frac{\pi}{12C}\cos\frac{\pi}{12}%
\]
we see that%
\[
\emph{l}(\gamma_{v}([0,r_{0}]))\leq\frac{1}{C}\arcsin\left(  \frac{\pi}%
{12}\cos\frac{\pi}{12}\right)
\]
and hence by the slope bound \eqref{simple_bound}, using that $v$ can point to
any direction and that $L$ is connected with no boundary points, we have
\begin{equation}
\left\vert \lambda_{i}\right\vert \leq\tan\left(  \arcsin\left(  \frac{\pi
}{12} \cos\frac{\pi}{12}\right)  \right)  =\frac{\frac{\pi}{12}\cos\frac{\pi
}{12}}{\sqrt{1-\left(  \frac{\pi}{12}\cos\frac{\pi}{12}\right)  ^{2}}}%
<\frac{\pi}{12}.
\end{equation}

\end{proof}

\subsection{Smoothness estimates}

The local graphical representation in Lemma \ref{lemma1} with bounded Hessian
for the Lagrangian phase function can be used to construct, by a rotation, a
new Lagrangian graph which lies in a region of the Hessian space where the
Lagrangian phase function is uniformly concave. Therefore, for Hamiltonian
stationary Lagrangian graphs, a priori $C^{2,\alpha}$ estimates apply to the
Lagrangian potential and then the bootstrapping procedure in \cite{ChenWarren}
leads to higher order estimates on a ball of uniform radius.

\begin{proposition}
\label{r1} Suppose that $L=\iota(M)$ is Hamiltonian stationary Lagrangian
given by a proper immersion $\iota$. Suppose that $\left\Vert A\right\Vert
_{\infty}\leq C$ on $L\cap B_{\rho_{0}}^{2n}(0)$ and $0\in L$ as in Lemma
\ref{lemma1}. Let $\Sigma$ be an embedded connected component of $B_{\rho_{0}%
}^{{2n}}(0)\cap L$ containing $0$ and let
\[
U_{\pi/6}(T_{0}^{n}\Sigma)=e^{-i\frac{\pi}{6}}I_{\mathbb{C}^{n}}(T_{0}%
^{n}\Sigma)
\]
where $e^{-i\frac{\pi}{6}}I_{\mathbb{C}^{n}}$ is the complex multiplication
acting on vectors in the (real) subspace $T_{0}^{n}\Sigma\subset\mathbb{C}%
^{n}$. Then, there exists $r_{1}(\left\Vert A\right\Vert _{\infty})$,
$C_{4}(\alpha,\left\Vert A\right\Vert _{\infty})$ such that $\Sigma$ is a
gradient graph over a region $\Omega\subset U_{\pi/6}(T_{0}^{n}\tilde{L});$
that is
\[
\Sigma=\left\{  (x,D\bar{u}(x)),x\in\Omega\right\}
\]
such that $B_{r_{1}}^{n}(0)\subset\Omega$ and we have that
\[
\left\Vert D^{4}\bar{u}\right\Vert _{C^{\alpha}\left(  B_{r_{1}}\right)  }\leq
C_{4}(\alpha,\left\Vert A\right\Vert _{\infty})\text{ on }B_{r_{1}}^{n}(0)
\]
with
\[
r_{1}=\frac{\pi\left(  1-4\sin^{2}\frac{\pi}{12}\right)  \cos\frac{\pi}{12}%
}{12\left\Vert A\right\Vert _{\infty}\times8}.
\]

\end{proposition}

\begin{proof}
First, by Lemma \ref{lemma1} we know that $\Sigma$ is represented by the
gradient graph of a function $u$ over a ball $B_{r_{0}}^{{n}}(0)$ contained in
the tangent space at $0$, and the Hessian of $u$ satisfies (\ref{HB}). As in
\cite[Proposition 4.1]{ChenWarren} we can use a Lewy-Yuan rotation
\cite[Section 2, Step 1]{YY06} to rotate the graph up by $\frac{\pi}{6}$:
\begin{align*}
\bar{x}  &  =\cos\frac{\pi}{6}\,x+\sin\frac{\pi}{6}\,Du(x)\\
\bar{y}  &  =-\sin\frac{\pi}{6}\,x+\cos\frac{\pi}{6}\,Du(x).
\end{align*}
Now by \cite[Proposition 4.1]{ChenWarren}, the graph of the gradient of the
new potential function
\begin{equation}
\bar{u}(x)=u(x)+\sin\frac{\pi}{6}\cos\frac{\pi}{6}\,\,\frac{\left\vert
Du(x)\right\vert ^{2}-\left\vert x\right\vert ^{2}}{2}-\sin^{2}\frac{\pi}%
{6}\,\,Du(x)\cdot x \label{u-bar}%
\end{equation}
over the $\bar{x}$-plane represents the same piece of $\Sigma$. It follows
that all of the eigenvalues now satisfy
\label{0}%
\begin{equation}
\lambda_{i}\in\left(  \tan\frac{\pi}{12},\tan\frac{\pi}{3}\right)  .
\label{HB2}%
\end{equation}
Thus the Lagrangian phase operator
\begin{equation}
F(D^{2}\varphi)=\sum_{\lambda_{j}\text{ eigenvalues of }D^{2}\varphi}%
\arctan\lambda_{j}\text{ } \label{LPO}%
\end{equation}
is uniformly concave on this region of Hessian space. We also know that the
Jacobian of the rotation map (cf. \cite[4.4]{ChenWarren}) is bounded below by
\begin{equation}
\det\frac{d\bar{x}}{dx}\geq\det\left[  \cos\frac{\pi}{6}\,\,I-\sin\frac{\pi
}{6}\tan\frac{\pi}{12}\,\,I\right]  >0.7. \label{JacobianB}%
\end{equation}
Thus the rotation of coordinates $x\rightarrow\bar{x}$ must give us a radius
\begin{equation}
\bar{r}_{0}=\left(  1-4\sin^{2}\frac{\pi}{12}\right)  r_{0}
\label{coordchange}%
\end{equation}
such that submanifold is graphical over a ball of radius $\bar{r}_{0},$ for a
new potential $\bar{u}$ representing the gradient graph over the plane
$U_{\pi/6}\left(  T_{0}\Sigma\right)  $. Now the Lagrangian phase operator
(\ref{LPO}) extends to a global (on Hessian space) concave uniformly elliptic
operator $\tilde{F}$ (cf. \cite[Section 5]{ChenWarren}) which agrees with $F$
on the following region of the Hessian space:
\[
\left\{  D^{2}\varphi:\left(  \tan\frac{\pi}{12}\right)  I\leq D^{2}%
\varphi\leq\left(  \tan\frac{\pi}{3}\right)  I\right\}  .
\]
In particular
\[
F(D^{2}\bar{u}\left(  \bar{x}\right)  )=\bar{\theta}(\bar{x})=\theta
(x)+n\frac{\pi}{6}.
\]

A rescaling of $\bar{u}$ gives $\bar{v}:$%
\[
\bar{v}(\bar{x})=\frac{1}{\bar{r}_{0}^{2}}\bar{u}(\bar{r}_{0}\bar{x})
\]
which is still a solution of the Hamiltonian stationary equation, since
(\ref{hseq}) only involves the second order derivatives of $\bar{u}$ which are
invariant under the rescaling, but now on the ball of radius $1$. Note that
the range of the Hessian (\ref{HB2}) does not change under rescaling, in
particular, if $\tilde{\theta}$ is the rescaled $\bar{\theta}$
\[
\tilde{\theta}(\bar{x})=\bar{\theta}\left(  \bar{r}_{0}\bar{x}\right)
\]
then $\tilde{\theta}$ satisfies a uniformly elliptic equation, with
ellipticity constants
\begin{align*}
\lambda_{0} =\frac{1}{1+\tan^{2}\frac{\pi}{3} },\,\,\,\,\, \Lambda_{0} =1
\end{align*}
according to (\ref{HB2}). Thus, by the De Giorgi-Nash-Moser theory, there is a
universal interior H\"{o}lder bound on $\tilde{\theta}$:
\[
\|\tilde{\theta}\|_{C^{\alpha}(B_{3/4})}\leq C_{DNM}\left(  \lambda
_{0},n\right)
\]
noting that $\bar{\theta}$ is bounded also by (\ref{HB2}).

Now we can apply \cite[Theorem 1.2]{CC03} \ to obtain
\begin{align*}
\Vert D^{2}\bar{v}\Vert_{C^{\alpha}(B_{1/2})}  &  \leq C_{CC}\left\{
\Vert\tilde{\theta}\Vert_{C^{\alpha}(B_{3/4})}+\Vert\bar{v}\Vert_{L^{\infty
}(B_{1})}\right\} \\
&  \leq C_{CC}\left\{  C_{DNM}\left(  \lambda_{0},n\right)  +\left\Vert
\bar{v}\right\Vert _{L^{\infty}(B_{1})}\right\}  .
\end{align*}
Now we also have%
\[
\left\Vert \bar{v}\right\Vert _{L^{\infty}(B_{1})}=\frac{1}{\bar{r}_{0}^{2}%
}\left\Vert \bar{u}\right\Vert _{L^{\infty}(B_{\bar{r}_{0}})}.
\]
We were assuming that $Du(0)=0,u\left(  0\right)  =0$ so that with (\ref{HB})
we have
\[
\left\Vert u\right\Vert _{L^{\infty}(B_{r_{0}})}\leq\tan\frac{\pi}{12}%
\frac{r_{0}^{2}}{2}%
\]
which leads to, by using \eqref{u-bar},
\[
\left\Vert \bar{u}\right\Vert _{L^{\infty}(B_{\bar{r}_{0}})}\leq\left\{
\tan\frac{\pi}{12}+\sin\frac{\pi}{6}\cos\frac{\pi}{6}\left(  2\tan^{2}%
\frac{\pi}{12}+1\right)  +\sin^{2}\frac{\pi}{6}\left(  \tan\frac{\pi}%
{12}+1\right)  \right\}  \frac{r_{0}^{2}}{2}.
\]
We conclude that
\[
\left\Vert D^{2}\bar{v}\right\Vert _{C^{\alpha}(B_{1/2})}\leq C_{CC}\left\{
C_{DNM}\left(  \lambda_{0},n\right)  +\frac{1}{\bar{r}_{0}^{2}}\,C_{T}\left(
\frac{\pi}{12}\right)  r_{0}^{2}\right\}
\]
for some universal trigonometric constant $C_{T}.$ \ Noting that bound
(\ref{coordchange}) bounds the ratio between $r_{0}$ and $\bar{r}_{0}$ we see
that we have a universal bound (depending only on $\alpha)$.
\[
\left\Vert D^{2}\bar{v}\right\Vert _{C^{\alpha}(B_{1/2})}\leq C_{2}\left(
\alpha\right)  .
\]
Now that the H\"{o}lder norm of $D^{2}\bar{v}$ is uniformly bounded, we may
apply the bootstrapping theory \cite[Section 5]{ChenWarren} to obtain
\begin{align*}
\left\Vert D^{3}\bar{v}\right\Vert _{C^{\alpha}(B_{1/4})}  &  \leq
C_{3}\left(  C_{2},\alpha\right) \\
\left\Vert D^{4}\bar{v}\right\Vert _{C^{\alpha}(B_{1/8})}  &  \leq
C_{4}\left(  C_{3},C_{2},\alpha\right)  .
\end{align*}
Now we may scale back to $\bar{u}$ and get that
\[
\left\Vert D^{4}\bar{u}\right\Vert _{C^{\alpha}(B_{\bar{r}_{0}/8})}\leq
C_{4}(\alpha)\bar{r}_{0}^{\,-2-\alpha}%
\]
Choosing $r_{1}=\bar{r}_{0}/8$ and recalling (\ref{coordchange}) and
(\ref{r_0}) gives us the result.
\end{proof}

\subsection{Curvature estimates with small total extrinsic curvature}

The next result establishes the key pointwise curvature estimates of a
Hamiltonian stationary submanifold under the assumption that the total
extrinsic curvature $\Vert A\Vert_{L^{n}}$ is small. This is an analogue to
results on minimal surfaces, harmonic maps and prescribed mean curvature
hypersurfaces (cf. \cite{CSc}, \cite{An}, \cite{ScUh}, \cite{Sharp},
\cite{ZhouZhu}). The main difference here from the minimal surfaces case is
the lack a useful Simons' type inequality in the Hamiltonian stationary case.
The $C^{4,\alpha}$ estimate for the scalar potential function $u$ allows us to
carry through an argument similar to that in \cite{CSc}.

\begin{proposition}
\label{e-regularity} \bigskip Suppose that $L$ is a smooth Lagrangian
Hamiltonian stationary manifold in $B_{1}(0)$ with $\partial L\cap
B_{1}(0)=\emptyset$. Then, there exists an $\varepsilon_{0}$ such that if
$r_{0}\leq1$ and
\[
\int_{B_{r_{0}}(0)\cap L}\left\vert A\right\vert ^{n}<\varepsilon_{0}
\]
then for all $0<\sigma\leq r_{0}$ and $y\in B_{r_{0}-\sigma}$
\[
\sigma^{2}\left\vert A(y)\right\vert ^{2}\leq\left(  \frac{\pi}{24}\right)
^{2}.
\]

\end{proposition}

\begin{proof}
Without loss of generality let $r_{0}=1$. We will deduce the general case by
rescaling at the end. Consider the nonnegative function $\left(  1-|x|\right)
^{2}|A(x)|^{2}. $ This function attains its maximum somewhere inside
$B_{1}(0)$, say at $x_{0}$. We assume the maximum is positive, otherwise the
result is trivial. Thus
\[
\left(  1-|x|\right)  ^{2}|A(x)|^{2}\leq\left(  1-|x_{0}|\right)  ^{2}%
|A(x_{0})|^{2}
\]
in particular, for $x\in B_{\frac{1-\left\vert x_{0}\right\vert }{2}}(x_{0})$
\begin{align*}
|A(x)|^{2}  &  \leq\frac{\left(  1-|x_{0}|\right)  ^{2}}{\left(  1-|x|\right)
^{2}}|A(x_{0})|^{2}\\
&  \leq\frac{\left(  1-|x_{0}|\right)  ^{2}}{\left(  \frac{1-\left\vert
x_{0}\right\vert }{2}\right)  ^{2}}|A(x_{0})|^{2}\\
&  =4|A(x_{0})|^{2}.
\end{align*}

Rescaling the graph over the ball $B_{\frac{1-\left\vert x_{0}\right\vert }%
{2}}(x_{0})$ by $|A(x_{0})|$, we get a Hamiltonian stationary manifold on a
ball of radius
\[
R_{0}=\frac{1-|x_{0}|}{2}\,|A(x_{0})|
\]
such that the second fundamental form $\tilde{A}$ satisfies
\[
|\tilde{A}(0)|=1\, \,\hbox{and} \,\,\, |\tilde{A}|\leq4.
\]

First, we suppose that (this will be contradicted)
\[
R_{0}>\frac{\pi}{48}%
\]
and
\[
\int_{B_{r_{0}}(0)\cap L}\left\vert A\right\vert ^{n}<\varepsilon_{0}.
\]
We have a Hamiltonian stationary Lagrangian submanifold on a ball of radius
$\frac{\pi}{48}$ with $| \tilde{A}| \leq4$. It follows that there is an
interior ball of radius $r_{1}(4)$ (from Lemma \ref{r1}) such that $L$ is
represented as the gradient graph of a function with
\begin{align*}
\left\Vert D^{2}u\right\Vert _{C^{\alpha}(B_{r_{1}})}  &  \leq\tan\frac{\pi
}{3},\\
\left\Vert D^{4}u\right\Vert _{C^{\alpha}(B_{r_{1}})}  &  \leq C_{4}(4).
\end{align*}
In particular, we have
\[
\left\|  \nabla\tilde{A}\right\|  _{C^{0}(B_{r_{1}})}\leq C_{5}.
\]
Therefore, as $|\tilde{A} (0)|=1$ we have
\[
| \tilde{A} | >\frac{1}{2} \,\,\,\,\,\hbox{on $B_{\frac{1}{2C_{5}}}(0)$.}
\]
Then integration leads to
\[
\int_{B_{\frac{1}{2C_{5}}}(0)}|\tilde{A}|^{n}\geq\left(  \frac{1}{2C_{5}%
}\right)  ^{n}\left(  \frac{1}{2}\right)  ^{n}=\frac{1}{4^{n}C_{5}^{n}}.
\]
Take
\[
\varepsilon_{0}=\frac{1}{4^{n}C_{5}^{n}}.
\]
So we have
\[
\int_{B_{\frac{1}{2C_{5}}}(0)}|\tilde{A}|^{n}\geq\varepsilon_{0}%
\]
which contradicts, by the scaling invariance of the total curvature, the
assumption
\[
\int_{B_{1}(0)}\left\vert A\right\vert ^{n}<\varepsilon_{0}.
\]
So we reject our assumption that $R_{0}>\frac{\pi}{48}$ and conclude that
\[
R_{0}\leq\frac{\pi}{48}.
\]
In this case, we have
\[
\frac{1-|x_{0}|}{2}|A(x_{0})|\leq\frac{\pi}{48}%
\]
which in turn implies
\[
\left(  1-|x|\right)  ^{2}|A(x)|^{2}\leq\left(  1-|x_{0}|\right)  ^{2}%
|A(x_{0})|^{2}\leq\left(  \frac{\pi}{24}\right)  ^{2}.
\]
It follows that, for $\left\vert x\right\vert \leq r$ we have%
\[
|A(x)|^{2}\leq\frac{1}{\left(  1-r\right)  ^{2}}\left(  \frac{\pi}{24}\right)
^{2}.
\]

Now suppose $r_{0}<1$. Rescaling the manifold by a factor of $\frac{1}{r_{0}}$
the first condition still holds, and we obtain
\[
r_{0}^{2}\left\vert A(x)\right\vert ^{2}=\left|  \tilde{A}\left(  \frac
{x}{r_{0}} \right)  \right|  ^{2}\leq\frac{1}{\left(  1-\frac{r}{r_{0}%
}\right)  ^{2}}\left(  \frac{\pi}{24}\right)  ^{2}.
\]
That is
\[
\left\vert A(x)\right\vert ^{2}\leq\frac{1}{\left(  r_{0}-r\right)  ^{2}%
}\left(  \frac{\pi}{24}\right)  ^{2}
\]
which is the conclusion.
\end{proof}

\section{Extension of Hamiltonian stationary Lagrangians across a small set}

\subsection{Extending Hamiltonian stationary sets under volume constrains}

The following extendibility result will be used in the proof of Theorem
\ref{conv_thm} to conclude the limiting varifold of a sequence of smooth
Hamiltonians stationary Lagrangian immersions is Hamiltonian stationary
including singular points; there, in fact we will only need the special case
that the singular set is of finitely many points.

\begin{theorem}
\label{ext_thm} \label{N} Let $N=\cup_{\alpha=1}^{\alpha_{0}}N_{\alpha}$ be a
finite union of compact sets $N_{\alpha}$ in a domain $\Omega\subset
{\mathbb{C}}^{n}$ where each $N_{\alpha}$ has finite $k_{\alpha}$-dimensional
Hausdorff measure with $k_{\alpha}\leq n-2$ and satisfies the local
$k_{\alpha}$-noncollapsing property
\begin{equation}
\inf_{x\in N_{\alpha}}{\mathscr H}^{k_{\alpha}}(N_{\alpha}\cap B_{\varepsilon
}(x))\geq C_{3}\varepsilon^{k_{\alpha}} \label{noncollapse}%
\end{equation}
for all $\varepsilon\in(0,\delta)$ for some $\delta$ and a constant $C_{3}>0$
independent of $\varepsilon$. Let $L$ be an immersed Lagrangian submanifold in
$\Omega\backslash N$ with $\overline{L}\backslash L\subseteq N$ such that
$(L,\mu_{L})$ is Hamiltonian stationary in $\Omega\backslash N$, where
$\mu_{L}={\mathscr H}^{n}\llcorner\, \beta$ is the measure on $L$ and $\beta$
is an $\mathbb{N}$-valued ${\mathscr H}^{n}$-integrable function on $L$. Assume

\begin{enumerate}
\item[(i)] $\displaystyle\int_{\Omega}|{{H}}|^{n}\hbox{d}\mu_{L}<C_{1}$, where
$H$ is the generalized mean curvature vector of $(L,\mu_{L})$;

\item[(ii)] There exists a positive constant $C_{4}$ such that for any open
set $E\subseteq L$
\[
\mu_{L}(E)\leq C_{4}\mathscr H^{n}(E);
\]

\item[(iii)] There exists a decreasing sequence $\varepsilon_{i}\rightarrow0$
such that
\[
\mathscr{H}^{n}(L\cap B_{\varepsilon_{i}}(y))<C_{2}\,\varepsilon
_{i}^{k_{\alpha}+\frac{n}{n-1}}%
\]
for all $y\in N_{\alpha}$ with $C_{2}$ independent of $y$.
\end{enumerate}

Then the closure $\overline{L}$ of $L$ is Hamiltonian stationary in $\Omega
$\emph{:} $\overline{L}$ admits a generalized mean curvature $\mathcal{H}$ in
$\Omega$ such that for any $f\in C_{0}^{\infty}(\Omega)$ it holds
\[
\int_{\Omega}\langle J\hbox{D}f,{\mathcal{H}}\rangle\ d\mu_{\overline{L}}=0.
\]

\end{theorem}

\begin{proof}
Define the $\varepsilon$-neighborhood of the compact set $N_{\alpha}$ by
\[
U_{\varepsilon}^{\alpha}=\{x\in{\mathbb{R}}^{2n}:\min_{y\in N_{\alpha}%
}|x-y|<\varepsilon\}.
\]
Then
\[
U_{\varepsilon}=\bigcup_{\alpha=1}^{\alpha_{0}}U_{\varepsilon}^{\alpha}%
\]
is the $\varepsilon$-neighborhood of $N$. Since $N$ is compact, we may assume
$U_{\varepsilon}$ is contained in the open domain $\Omega$ by choosing
$\varepsilon$ small. For simplicity of notations, we will assume (iii) holds
for $3\varepsilon_{i}$'s.

\medskip

\textit{\textbf{Step 1.} Volume estimate of $L\cap U_{\varepsilon_{j}}$.}

\vspace{.15cm}

For any fixed large $j$, let $\{B_{\varepsilon_{j}}(x^{\alpha}_{1}%
),...,B_{\varepsilon_{j}}(x^{\alpha}_{{\ell}(\varepsilon_{j})})\}$ be the
maximal family of disjoint balls in $\Omega\subset\mathbb{R}^{2n}$ centered at
$x^{\alpha}_{i}\in N_{\alpha}$ of radius $\varepsilon_{j}$. Compactness of
$N_{\alpha}$ ensures the number $\ell_{\alpha}(\varepsilon_{j})$ well defined.
The maximality assumption then implies
\[
N_{\alpha}\subseteq\bigcup_{i=1}^{\ell_{\alpha}(\varepsilon_{j})}%
B_{2\varepsilon_{j}}(x^{\alpha}_{i}).
\]
To estimate $\ell_{\alpha}(\varepsilon_{j})$, summing the $k_{\alpha}%
$-dimensional Hausdorff measures over the disjoint balls and using the local
$k_{\alpha}$-noncollapsing assumption \eqref{noncollapse}, we have
\[
\ell_{\alpha}(\varepsilon_{j})C_{3} \varepsilon_{j}^{k_{\alpha}}\leq\sum
_{i=1}^{\ell_{\alpha}(\varepsilon_{j})}\mathscr{H}^{k_{\alpha}}(N_{\alpha}\cap
B_{\varepsilon_{j}}(x^{\alpha}_{i}))\leq\mathscr{H}^{k_{\alpha}}(N_{\alpha})
\]
Therefore
\[
\ell_{\alpha}(\varepsilon_{j})\leq\frac{\mathscr{H}^{k_{\alpha}}(N)}{C_{3}
\,\varepsilon_{j}^{k_{\alpha}}}.
\]

Next, we claim
\[
U_{\varepsilon_{j}}^{\alpha}\subset\bigcup_{i=1}^{\ell_{\alpha}(\varepsilon
_{j})}B_{3\varepsilon_{j}}(x_{i}^{\alpha}).
\]
This can be seen from that for any point $p\in U_{\varepsilon_{j}}^{\alpha}$
there is a point $q\in N_{\alpha}$ with $|p-q|\leq\varepsilon_{j}$ and $q\in
B_{2\varepsilon_{j}}(x_{i}^{\alpha})$ for some $i$, and it follows $p\in
B_{\varepsilon_{j}}(x_{i}^{\alpha})$. Now by the assumptions (ii) and (iii),
\begin{align}
\int_{U_{\varepsilon_{j}}}\hbox{d}\mu_{L}  &  \leq\sum_{\alpha=1}^{\alpha_{0}%
}\int_{U_{\varepsilon_{j}}^{\alpha}}\hbox{d}\mu_{L}\nonumber\\
&  \leq\sum_{\alpha=1}^{\alpha_{0}}\sum_{i=1}^{\ell_{\alpha}(\varepsilon_{j}%
)}\int_{B_{3\varepsilon_{j}}(x_{i}^{\alpha})}\hbox{d}\mu_{L}\nonumber\\
&  \leq\sum_{\alpha=1}^{\alpha_{0}}\ell_{\alpha}(\varepsilon_{j})\,C_{4}%
C_{2}\,(3\varepsilon_{j})^{k_{\alpha}+\frac{n}{n-1}}\label{volume estimate}\\
&  \leq\sum_{\alpha=1}^{\alpha_{0}}\frac{\mathscr{H}^{k_{\alpha}}(N_{\alpha}%
)}{C_{3}}\,C_{4}C_{2}\,3^{k_{\alpha}+\frac{n}{n-1}}\,\varepsilon_{j}^{\frac
{n}{n-1}}\nonumber\\
&  =C_{5}(N)\,\varepsilon_{j}^{\frac{n}{n-1}}.\nonumber
\end{align}

\textit{\textbf{Step 2.} Existence of the generalized mean curvature
$\mathcal{H}$ of $\overline{L}$ in $\Omega$. }

\vspace{.15cm}

Let $X$ be an arbitrary $C^{1}$ vector field on $\Omega$ with compact support.
Our goal is to verify \cite[Definition 16.5]{Simon}
\begin{equation}
\int_{\Omega}\mbox{div}_{\overline{L}}X\, \mbox{d}\mu_{\overline L}%
=-\int_{\Omega}\langle\mathcal{H},X\rangle\,\mbox{d}\mu_{\overline L}
\label{globalbar}%
\end{equation}
for some locally $\mu_{\overline L}$ -integrable $\mathbb{R}^{2n}$-valued
function $\mathcal{H}$ on $\overline{L}$.

Let $\phi_{\varepsilon_{j}}$ be a cut-off function satisfying
\begin{equation}
\left\{
\,\,\begin{aligned} &\phi_{\varepsilon_{j}}=0\,\,\,\hbox{on $U_{\varepsilon_j/2}$} \nonumber \\ &\phi_{\varepsilon_{j}}=1\,\,\,\hbox{on $\Omega\backslash U_{\varepsilon_j}$} \nonumber \\ &0\leq\phi_{\varepsilon_j} \leq1 \nonumber \\ & |D\phi_{\varepsilon_j}|<C/{\varepsilon_{j}}. \nonumber \end{aligned}\right.
\end{equation}
The existence of such $\phi_{\varepsilon_{j}}$ is given, for example, in Lemma
2.2 in \cite{Harvey-Polking} and is also due to Bochner \cite{Bochner}. Then
$\phi_{\varepsilon_{j}}X$ is a $C^{1}$ vector field which vanishes on
$U_{{\varepsilon_{j}}/2}$. By the standard first variation formula, we have%

\begin{align}
\int_{\Omega}  &  \langle H,\phi_{\varepsilon_{j}}X\rangle\ \hbox{d}\mu_{L}
=-\int_{\Omega}\mbox{div}_{L}(\phi_{\varepsilon_{j}}X) \ \hbox{d}\mu
_{L}\label{general H}\\
&  =-\int_{\Omega} \left\{  \langle\nabla\phi_{\varepsilon_{j}},X\rangle+
\phi_{\varepsilon_{j}}\mbox{div}_{L}X\right\}  \ \hbox{d}\mu_{L}.\nonumber
\end{align}
From the volume estimate \eqref{volume estimate},
\[
\left\vert \int_{\Omega}\langle\nabla\phi_{\varepsilon_{j}},X\rangle
\ \hbox{d}\mu_{L}\right\vert \leq C(X)\, {\varepsilon_{j}}^{-1}\int%
_{U_{{\varepsilon_{j}}}\backslash U_{{\varepsilon_{j}}/2}} \hbox{d}\mu
_{L}\rightarrow0.
\]
Now letting ${\varepsilon_{j}}\rightarrow0$ in \eqref{general H}
\begin{equation}
\int_{\Omega}\langle H,X\rangle\ \hbox{d}\mu_{L}=-\int_{\Omega} \mbox{div}_{L}%
X\ \hbox{d}\mu_{L}. \label{global}%
\end{equation}

By assumption, $\overline{L}\backslash L\subseteq N$ and $\mathscr{H}^{k}%
(N)<+\infty$ and $k\leq n-2,$ we have $\mathscr{H}^{n}\left(  \overline
{L}\backslash L\right)  =0$. So $\overline{L}=L\cup(\overline{L}\backslash L)$
is a rectifiable $n$-varifold. The divergence operator $\mbox{div}_{\overline
{L}}$ is defined as $\mbox{div}_{L}$, by noting that $\overline{L}\backslash
L$ has zero measure (cf. \cite[16.2]{Simon}). Then by \eqref{global}
\begin{align}
\int_{\Omega}\mbox{div}_{\overline{L}}X\ \hbox{d}\mu_{\overline L}  &
=\int_{\Omega}\hbox{div}\, X \ \hbox{d}\mu_{L}\nonumber\\
&  =-\int_{\Omega}\langle{H},X\mathcal{\rangle} \ \hbox{d}\mu_{L}%
\label{global2}\\
&  =-\int_{\Omega}\langle\mathcal{H},X\mathcal{\rangle}\ \hbox{d}\mu
_{\overline L}\nonumber
\end{align}
where $\mathcal{H}$ equals ${H}$ on $L$ and zero on $\overline{L}\backslash
L$, so it is locally $\mu_{L}$-integrable on $\overline{L}$, in turn
$\mathcal{H}$ is the generalized mean curvature of $\overline{L}$ in $\Omega$
since $X$ is arbitrary.

\vspace{0.2cm}

\textit{\textbf{Step 3.} $\overline{L}$ is Hamiltonian stationary in $\Omega
$.}

\vspace{.15cm}

Our goal is to show that
\begin{equation}
\int_{\Omega}\langle J\hbox{D}f,{\mathcal{H}}\rangle\,\hbox{d}\mu
_{\overline{L}}=0\label{globalHSdone}%
\end{equation}
for all $f\in C_{0}^{\infty}\left(  \Omega\right)  $. For any smooth function
$f$ with compact support in $\Omega$, $J\hbox{D}(\phi_{\varepsilon_{j}}f)$ is
a Hamiltonian vector field on $\Omega$ with compact support, in particular it
vanishes on $U_{{\varepsilon_{j}}/2}$ containing $N$. Applying \eqref{global2}
with $X=J\nabla f$, we see
\begin{align}
\int_{\Omega}\langle J\hbox{D}f,{\mathcal{H}}\rangle\,\hbox{d}\mu_{L} &
=\int_{L}\langle J\nabla f,{H}\rangle\,\hbox{d}\mu_{L}\nonumber\\
&  =\int_{L\cap U_{\varepsilon_{j}}}\langle J\nabla f,{H}\rangle
\,\hbox{d}\mu_{L}+\int_{L\backslash U_{\varepsilon_{j}}}\langle J\nabla
f,{H}\rangle\,\hbox{d}\mu_{L}.\label{stationary2show}%
\end{align}

Since $L$ is Hamiltonian stationary in $\Omega\backslash N$, we have%
\begin{align}
\left\vert \int_{L\backslash U_{\varepsilon_{j}}}\langle J\nabla f,{H}%
\rangle\,\hbox{d}\mu_{L}\right\vert  &  =\left\vert \int_{L}\langle
J\nabla(\phi_{\varepsilon_{j}}f),H\rangle\,\hbox{d}\mu_{L}-\int_{L\cap
U_{\varepsilon_{j}}}\langle J\nabla(\phi_{\varepsilon_{j}}f),H\rangle
\,\hbox{d}\mu_{L}\right\vert \nonumber\\
&  =\left\vert \,0-\int_{L\cap(U_{\varepsilon_{j}}\backslash U_{{\varepsilon
_{j}}/2})}\left(  \langle\phi_{\varepsilon_{j}}J\nabla f,H\rangle+\langle
fJ\nabla\phi_{\varepsilon_{j}},H\rangle\right)  \,\hbox{d}\mu_{L}\right\vert
\nonumber\\
&  \leq C(f)(1+{\varepsilon_{j}}^{-1})\int_{L\cap(U_{\varepsilon_{j}}\setminus
U_{{\varepsilon_{j}}/2})}|{H}|\,\hbox{d}\mu_{L}\nonumber\\
&  \leq C(f)(1+{\varepsilon_{j}}^{-1})\left(  \int_{L\cap(U_{\varepsilon_{j}%
}\backslash U_{{\varepsilon_{j}}/2})}|{H}|^{n}\,\hbox{d}\mu_{L}\right)
^{\frac{1}{n}}\left(  \int_{U_{\varepsilon_{j}}\backslash U_{{\varepsilon_{j}%
}/2}}\hbox{d}\mu_{L}\right)  ^{\frac{n-1}{n}}\label{bound2zero1}%
\end{align}
by H\"{o}lder's inequality, where $C(f)$ depends on $f$ and $|\hbox{D}f|$ as
$\nabla f$ is the tangential projection of $\hbox{D}f$ along $L$ so
\[
|J\nabla f|=|\nabla f|\leq|\hbox{D}f|.
\]
Similarly
\begin{equation}
\left\vert \int_{L\cap U_{\varepsilon_{j}}}\langle J\nabla f,{H}%
\rangle\,\hbox{d}\mu_{L}\right\vert \leq C(f)\left(  \int_{L\cap
U_{\varepsilon_{j}}}|{H}|^{n}\,\hbox{d}\mu_{L}\right)  ^{\frac{1}{n}}\left(
\int_{U_{{\varepsilon_{j}}}}\hbox{d}\mu_{L}\right)  ^{\frac{n-1}{n}%
}\label{bound2zero2}%
\end{equation}

It then follows from the assumption (i), and the volume estimate
\eqref{volume estimate} that both terms (\ref{bound2zero1}) and
(\ref{bound2zero2}) vanish as ${\varepsilon_{j}}\rightarrow0$. Combining with
(\ref{stationary2show}) we conclude (\ref{globalHSdone}).
\end{proof}

The local $k$-noncollapsing property is automatically satisfied if $N$ is a
compact manifold of dimension no larger than $n-2$.

\begin{corollary}
\label{removable-manifold} Let $N$ be a compact submanifold in a domain
$\Omega\subset{\mathbb{R}}^{2n}$ of dimension $k\leq n-2$. Let $L$ be
Hamiltonian stationary in $\Omega\backslash N$ as in Theorem \ref{ext_thm}
with (i) and (ii) therein. Then $\overline{L}$ is Hamiltonian stationary in
$\Omega$.
\end{corollary}

\begin{corollary}
With the assumptions on $N$ and (i), (ii) as in Theorem \ref{ext_thm}, let
$\iota:M\rightarrow\Omega\backslash N$ be a proper immersion of an
$n$-dimensional manifold $M$ in $\Omega\backslash N$ and $L=\iota(M)$ is
Hamiltonian stationary Lagrangian in $\Omega\backslash N$. Then $\overline{L}$
is Hamiltonian stationary Lagrangian in $\Omega$.
\end{corollary}

\begin{proof}
In light of Theorem \ref{ext_thm}, the only thing to verify is: $\overline
{L}\backslash L\subseteq N$. For any $y\in\overline{L}\backslash L$, if
$y\not \in N$ then by compactness of $N$ there will be a neighborhood $W$ of
$y$ such that $\overline{W}\cap N=\emptyset$; then there exists a sequence
$y_{j}\in\overline{W}\cap L\rightarrow y$. By properness of $\iota$, it
follows that $\iota^{-1}(\{y_{j}:j\in\mathbb{N}\})$ contains a converging
subsequence in $M$ since $\iota^{-1}(\overline{W})$ is compact in $M$; then
$y$ is the image of the limit point which is in $L$, and we have a contradiction.
\end{proof}

\subsection{Volume estimate via the monotonicity formula}

The following volume upper estimate is a direct consequence of the standard
monotonicity formula for volumes. In particular, it implies that the
assumption (iii) in Theorem \ref{ext_thm} holds when $N$ is a finite set of
points ($k=0$) and $H\in L^{n}$ when we take the Radon measure $\mu$ induced
by $\mathscr H^{k}$ (or an finite integral multiple of $\mathscr H^{k}).$

\begin{proposition}
\label{points} Let $L$ be an integral $n$-rectifiable varifold in
$\mathbb{R}^{n+l}$, with generalized mean curvature $\mathcal{H}$ in
$L^{n}(L,\mu)$ where $\mu$ is the Radon measure associated with $L$. Then
\[
\mu(B_{\rho}(x))\leq C\left(  \left\vert \ln\rho\right\vert +1\right)
^{n}\rho^{n}.
\]
In particular when $n\geq2$, for any $0\leq k\leq n-2$ it holds for small
$\rho$
\[
\mu(B_{\rho}(x))\leq C\rho^{k+\frac{n}{n-1}}.
\]

\end{proposition}

\begin{proof}
Recall the monotonicity formula \cite[17.3 p. 84]{Simon}
\begin{align}
\frac{d}{d\rho}\left(  \rho^{-n}\mu(B_{\rho}(x))\right)   &  =\frac{d}{d\rho
}\int_{B_{\rho}(x)}\frac{|\hbox{D}^{\perp}r|^{2}}{r^{n}}\mbox{d}\mu
+\rho^{-1-n}\int_{B_{\rho}(x)}\langle y-x,\mathcal{H}\rangle\mbox{d}\mu
\label{monotonicity}\\
&  \geq\rho^{-1-n}\int_{B_{\rho}(x)}\langle y-x,\mathcal{H}\rangle
\mbox{d}\mu\nonumber\\
&  \geq-\rho^{-1-n}\int_{B_{\rho}(x)}\rho\left\vert \mathcal{H}\right\vert
\mbox{d}\mu\nonumber\\
&  \geq-\rho^{-n}\left(  \int_{B_{\rho}(x)}\left\vert \mathcal{H}\right\vert
^{n}\mbox{d}\mu\right)  ^{1/n}\mu(B_{r}(x))^{\frac{n-1}{n}}.\nonumber
\end{align}
Now let
\[
w(\rho)=\frac{\mu(B_{\rho}(x))^{1/n}}{\rho}%
\]
in which case we have
\[
\frac{d}{d\rho}\left[  w(\rho)\right]  ^{n}\geq-\frac{1}{\rho}\left(
\int_{B_{\rho}(x)}\left\vert \mathcal{H}\right\vert ^{n}\mbox{d}\mu\right)
^{1/n}w^{n-1}%
\]
and
\begin{align*}
nw^{n-1}\frac{d}{d\rho}w &  \geq-\frac{1}{\rho}\left(  \int_{B_{\rho}%
(x)}\left\vert \mathcal{H}\right\vert ^{n}\mbox{d}\mu\right)  ^{1/n}w^{n-1}\\
\frac{d}{d\rho}w &  \geq-\frac{1}{\rho n}\left(  \int_{B_{\rho}(x)}\left\vert
\mathcal{H}\right\vert ^{n}\mbox{d}\mu\right)  ^{1/n}.
\end{align*}
Integrating over $(\rho,\rho_{0})$,%
\[
w(\rho_{0})-w(\rho)\geq\left(  \int_{B_{\rho}(x)}\left\vert \mathcal{H}%
\right\vert ^{n}\mbox{d}\mu\right)  ^{1/n}\frac{1}{n}\left[  \ln\rho-\ln
\rho_{0}\right]
\]
that is
\[
w(\rho)\leq w(\rho_{0})+\left(  \int_{B_{\rho}(x)}\left\vert \mathcal{H}%
\right\vert ^{n}\mbox{d}\mu\right)  ^{1/n}\frac{1}{n}\left(  -\ln\rho+\ln
\rho_{0}\right)
\]
or
\[
\frac{\mu(B_{\rho}(x))}{\rho^{n}}\leq\left\{  \mu(B_{\rho_{0}}(x))+\left(
\int_{B_{\rho}(x)}\left\vert \mathcal{H}\right\vert ^{n}\mbox{d}\mu\right)
^{1/n}\frac{1}{n}\left(  \left\vert \ln\rho\right\vert +\ln\rho_{0}\right)
\right\}  ^{n}%
\]
and finally%
\[
\mu(B_{\rho}(x))\leq\rho^{n}\left\{  \mu(B_{\rho_{0}}(x))+\left(
\int_{B_{\rho}(x)}\left\vert \mathcal{H}\right\vert ^{n}\mbox{d}\mu\right)
^{1/n}\frac{1}{n}\left(  \left\vert \ln\rho\right\vert +\ln\rho_{0}\right)
\right\}  ^{n}.
\]
In particular we have
\[
\rho^{-k-\frac{n}{n-1}}\mu(B_{\rho}(x))\leq\rho^{\frac{n(n-2)}{n-1}-k}\left\{
\mu(B_{\rho_{0}}(x))+\left(  \int_{B_{\rho}(x)}\left\vert \mathcal{H}%
\right\vert ^{n}\mbox{d}\mu\right)  ^{1/n}\frac{1}{n}\left(  \left\vert
\ln\rho\right\vert +\ln\rho_{0}\right)  \right\}  ^{n}%
\]
and the term on the right hand side tends to zero as $\rho\rightarrow0$ when
$n>2$, as $k\leq n-2$ by assumption; however, when $n=2$, this term becomes unbounded.

For $n=2$, $k$ must be 0, and the desired result follows from \cite[(A.6)]%
{KSc} (cf. \cite{Simon93}): for any $0<\rho<\rho_{0}$,
\[
\rho^{-2}\mu(B_{\rho}(x))\leq C\rho_{0}^{-2}\mu(B_{\rho_{0}}(x)) + C
\int_{B_{\rho_{0}}(x)} |\mathcal{H} |^{2} \mbox{d}\mu<\infty.
\]

\end{proof}

\section{Sequential convergence of Hamiltonian stationary Lagrangians}

Convergence of a sequence of embedded manifolds in $C^{k}$ topology has been
used in \cite{CSc} and then in \cite{An} and recently in \cite{Sharp} via
local graphical representations of the manifolds. Along the same lines, we
write down a definition of $C^{k}$ convergence of manifolds to a varifold that
will be sufficient for our purposes.


\begin{definition}
\label{def1} \emph{Let $\left\{  S_{j}\right\}  $ be a sequence of finite sets
of embedded $n$-dimensional submanifolds $\left\{  \Sigma_{j,i}\right\}  $ in
an open subset $U$ of $\mathbb{R}^{2n},$ where $S_{j}=\left\{  \Sigma
_{j,1},...,\Sigma_{j,m}\right\}  $ \ for some positive integer $m$. Suppose
that for each $i\in\left\{  1,...,m\right\}  $ there is a point $x(i)$ and an
n-plane $P(i)$ containing $x(i)$ such that $\Sigma_{j,i}$ is a sequence of
graphs over $P(i).$ \ If (up to possible permutations of $\left\{
1,...,m\right\}  )$ each sequence of graphs converges uniformly in the $C^{k}$
topology to a graph $\Sigma_{\infty,i}$ over $P(i)$ we say that} $\left\{
S_{j}\right\}  $ converges uniformly in $C^{k}$ topology to the integral
varifold on $U$%
\[
V=\sum_{i=1}^{m}\Sigma_{\infty,i}%
\]
\emph{identifying each embedded submanifold $\Sigma_{\infty,i}$ with a
(multiplicity 1) integral varifold in the obvious way.}
\end{definition}

\begin{definition}
\label{def2} \emph{Given an open set $U$ in ${\mathbb{R}}^{2n}$, we say that}
a sequence of immersed submanifolds $L_{j}$ in $U$ converges uniformly to a
varifold $V$ in the $C^{k}$ topology in $U$, \emph{if for every point
$x\in\mbox{supp}\,(V)$ there is a neighborhood $U_{x}$ in $U$ such that
\[
S_{j}=\left\{  \hbox{embedded connected components of}\,\,L_{j}\cap
U_{x}\right\}
\]
converges uniformly in $C^{k}$ topology to $V$ restricted to $U_{x}$. }
\end{definition}

\begin{proof}
[Proof of Theorem 1.1]The case $n=1$ was discussed in the introduction. We now
assume $n\geq2$. First, we claim that the manifolds $L_{i}$ remain in a
bounded region in $\mathbb{C}^{n}$. Fixing an $L_{i},$ by the Wiener Covering
Lemma \cite[Lemma 4.1.1]{KrantzParks}, we may choose a finite collection of
balls $B_{1}(x_{k})$, for $x_{k}\in L_{i}$ that cover $L_{i}$ such that
$B_{1/3}(x_{k})$ are disjoint. Now for each $x_{k}$ either
\begin{equation}
\int_{B_{1/3}(x_{k})}\left\vert A\right\vert ^{n}<\varepsilon_{0} \label{altA}%
\end{equation}
or
\begin{equation}
\int_{B_{1/3}(x_{k})}\left\vert A\right\vert ^{n}\geq\varepsilon_{0}.
\label{altB}%
\end{equation}
In the first case, by Lemma \ref{e-regularity}, we have a uniform bound on the
curvature on $B_{1/6}(x_{k})$:
\[
\left\vert A\right\vert \leq\sqrt{\frac{3\pi}{2}.}%
\]
Lemma \ref{lemma1} then guarantees that there is a fixed minimum radius
\[
r_{1}=\frac{\sqrt{\pi}}{12}\sqrt{\frac{2}{3}}\cos\left(  \frac{\pi}%
{12}\right)
\]
such that a connected component of $L_{i}$ $\cap B_{r_{1}}(x_{k})$ is
graphical is over the tangent plane at $x_{k},$ which implies
\[
\hbox{Vol}\,(B_{1/6}(x_{k})\cap L_{i})\geq\omega_{n}r_{1}^{n}.
\]
It follows that the number of points $x_{k}$ for which (\ref{altA}) hold is
bounded by
\begin{equation}
\#\left\{  x_{k}:\int_{B_{1/3}(x_{k})}\left\vert A\right\vert ^{n}%
<\varepsilon_{0}\right\}  \leq\frac{C_{1}}{\omega_{n}r_{1}^{n}}.
\label{numberofballs}%
\end{equation}
On the other hand, it is clear that
\[
\#\left\{  x_{k}:\int_{B_{1/3}(x_{k})}\left\vert A\right\vert ^{n}%
\geq\varepsilon_{0}.\right\}  \leq\frac{C_{2}}{\varepsilon_{0}}.
\]
It follows that there are at most
\[
R_{0}=\frac{C_{1}}{\omega_{n}r_{1}^{n}}+\frac{C_{2}}{\varepsilon_{0}}%
\]
balls of radius $1$ in this cover. Immediately we conclude (recall $L_{i}$ are
connected):
\[
L_{i}\subset B_{R_{0}}(0)=\{x\in{\mathbb{R}}^{2n}:|x|\leq R_{0}\}.
\]

Next, define
\[
{\mathcal{C}}_{k}=\{B_{r_{k}}(y_{k,j})\}
\]
to be a finite cover of $B_{R_{0}}(0)$ by balls $B_{r_{k}}(y_{k,j})$ in
$\mathbb{R}^{2n}$, where $r_{k}=2^{-k}\varepsilon_{0}$ and $\varepsilon_{0}$
is the constant in Proposition \ref{e-regularity}, with the property that each
point in $B_{R_{0}}(0)$ is covered by at most $b$ balls in ${\mathcal{C}}_{k}$
and $\{B_{r_{k}/2}(y_{k,j})\}$ still covers $B_{R_{0}}(0)$. This can be done
with $b$ independent of $r_{k},y_{k,j}$, by Besicovitch's covering theorem
(cf. \cite[Theorem 4.2.1]{KrantzParks}). Now we observe
\[
\sum_{j}\int_{L_{i}\cap B_{r_{k}}(y_{k,j})}|A_{i}|^{n}\hbox{d}\mu_{i}\leq
b\int_{L_{i}}|A_{i}|^{n}\hbox{d}\mu_{i}\leq b\,C_{2}%
\]
where $L_{i}\cap B_{r_{k}}(y_{k,j})\not =\emptyset$. It then follows that for
each $i$ and each $k$ there are $J_{k}^{i}$ balls of radius $r_{k}$ such that
the integral of $|A_{i}|^{n}$ on each of these balls is not smaller than
$\varepsilon_{0}$, for an integer $J_{k}^{i}$ with $J_{k}^{i}\leq
b\,C/\varepsilon_{0}$. By reindexing, we may denote the centers of these balls
by $y_{k,j}(i)$ and the collection of balls as
\begin{equation}
\mathcal{B}_{k}(i)=\{B_{r_{k}}(y_{k,1}(i)),\cdots,B_{r_{k}}(y_{k,J_{k}^{i}%
}(i))\}. \label{set_of_bad_balls}%
\end{equation}
Letting
\[
J_{k}=\lim\sup_{i\rightarrow\infty}J_{k}^{i}\leq\frac{b\,C}{\varepsilon_{0}}%
\]
we may choose a subsequence $\left\{  L_{i}\right\}  $ (here and in the
sequel, we will use the same indices for subsequences for simplicity) such
that $J_{k}^{i}=J_{k}$ for all $i$. We may then assume, by switching to a
subsequence if necessary, the sequence $y_{k,j}(i)\rightarrow x_{k,j}$ as
$i\rightarrow\infty$ for each $1\leq j\leq J_{k}^{i}\leq b\,C_{2}%
/\varepsilon_{0}$. Next, letting
\[
J=\lim\sup_{k\rightarrow\infty}J_{k}\leq\frac{b\,C_{2}}{\varepsilon_{0}}%
\]
we may select a subsequence $\mathcal{K}\subset\mathbb{N}$ such that
$\left\vert \mathcal{K}\right\vert =\infty$ and $J_{k}^{i}=J$ for all $i$ and
$k\in\mathcal{K}$. By choosing yet another subsequence we further assume that
$x_{k,j}\rightarrow x_{j}$ for each $j=1,...,J$ as $k\in\mathcal{K}%
\rightarrow\infty$, and let $S=\{x_{1},...,x_{J}\},$ and $S$ may be empty.

We assume there is no subsequence of $\{L_{i}\}$ that converges to a single
point, otherwise we are done. We construct a sequence of nested open sets
\[
U_{0}\subset U_{1}...\subset B_{R_{0}}\backslash S
\]
such that
\[
\bigcup_{l} \,U_{l}=B_{R_{0}}\backslash S
\]
and show that there is a subsequence $\{L_{i}\}$ that converges in $C^{m}$ in
the sense of Definition \ref{def2}, uniformly on each $U_{l}$ to a Hamiltonian
stationary varifold.

Let $\tau_{0}>0$ be smaller than the minimum distance between points in $S$
and the minimum distance from points in $S$ to $\partial B_{R_{0}}$ and let
$\tau_{l+1}=3^{-l}\tau_{1}.$ \ For each $l$, choose $k=k(l)$ $\in\mathcal{K}$
so that
\begin{align}
\left\Vert x_{k,j}-x_{j}\right\Vert  &  <\frac{\tau_{l}}{4},\,\,\text{ for all
}j\in\left\{  1,...,J\right\} \label{k1}\\
r_{k}  &  <\frac{\tau_{l}}{8}. \label{k2}%
\end{align}
In particular the balls $B_{\tau_{l}/2}(x_{k,j})$ are disjoint and contained
in $B_{\tau_{l}}(x_{j})$ respectively. Let
\begin{equation}
U_{l}=B_{R_{0}}(0)\backslash\bigcup_{x_{j}\in S}\overline{B_{\tau_{l}}(x_{j}%
)}. \label{U_l}%
\end{equation}
For a fixed $l$, we may choose $i\geq i(l)$ large enough so that
\begin{equation}
\left\Vert y_{k,j}(i)-x_{k,j}\right\Vert <\frac{\tau_{l}}{4}. \label{i}%
\end{equation}
It then follows that
\[
U_{l}\subset B_{R_{0}}(0)\setminus\bigcup_{y_{k,j}\in\mathcal{B}_{k}%
(i)}B_{r_{k}}(y_{k,j}).
\]
In particular, for each $i$ the set $U_{l}$ is covered by the balls
${\mathcal{C}}_{k}\backslash\mathcal{B}_{k}$, recall (\ref{set_of_bad_balls}),
and $d(\overline{U_{l}},S)\geq3\tau_{l}/8$. Then, for a ball $B_{r_{k}%
}(y_{k,j})$ with $L_{i}\cap U_{l}\cap B_{r_{k}}(y_{k,j})\not =\emptyset$, we
conclude that $y_{k,j}\notin\mathcal{B}_{k}(i)$, thus
\[
\Vert A_{i}\Vert_{L^{n}(L_{i}\cap B_{r_{k}}(y_{k,j}))}^{n}<\varepsilon_{0}%
\]
and then we have a curvature bound
\begin{equation}
\left\Vert A\right\Vert (y)\leq\frac{3\times2^{k}}{\varepsilon_{0}}\frac{\pi
}{24}%
\end{equation}
for points $y\in$ $L_{i}\cap B_{2r_{k}/3}(y_{k,j})$. This must hold uniformly
at each point of $U_{l}$. Now consider the components of $L_{i}\cap U_{l}\cap
B_{r_{k}}(y_{k,j})$ that intersect $B_{r_{k}/2}(y_{k,j})$. There are a finite
number of these, by the same reasoning leading to \eqref{numberofballs}.
Applying Lemma \ref{lemma1}, that for any point in one of these components,
the manifold stays graphical over a ball in the tangent plane of radius
\[
\frac{2\varepsilon_{0}}{3\times2^{k}}\cos\frac{\pi}{12}>\frac{r_{k}}{2}%
\]
with Lagrangian potential $u$ satisfying%
\begin{equation}
\left\vert D^{2}u\right\vert \leq\tan\frac{\pi}{12}\,\,\text{ on }B_{r_{k+1}%
}(y_{k,j}). \label{grassman}%
\end{equation}

Every embedded connected component of $L_{i}\cap B_{r_{k}/2}(y_{k,j})$ is
contained in an embedded connected component of $L_{i}\cap B_{r_{k}}(y_{k,j}%
)$. We may choose a subsequence of $\left\{  L_{i}\right\}  $ so that for each
$j$ the number $m(y_{k,j})$ of components of $L_{i}\cap B_{r_{k}}(y_{k,j})$
that intersect $B_{r_{k+1}}(y_{k,j})$ with $y_{k,j}\not \in \mathcal{B}%
_{k}(i)$, is independent of $i$, again by the same reasoning leading to
\eqref{numberofballs}. For such chosen $L_{i}$, each embedded connected
component of $L_{i}\cap B_{r_{k+1}}(y_{k,j})$ is graphical over an $n$-plane
in the Lagrangian Grassman, so using \eqref{grassman} we may choose a further
subsequence such that each sequence of components remains graphical over a
fixed Lagrangian $n$-plane. The bound (\ref{grassman}) together with
Proposition \ref{r1} gives uniform $C^{m}$ bounds for each graphing function
for each positive integer $m$; by Arzel\`{a}-Ascoli theorem, the graphs
converge uniformly to a limit. We therefore conclude that $\{L_{i}\cap
U_{l}\}$ converges uniformly in $C^{m}$ to a varifold (or vacates $U_{l}$
completely) in the sense of Definition \ref{def2}, and the limit is locally
the sum of finitely many immersed submanifolds, possibly with multiplicity.
Because every compact set $K\subset B_{R_{0}}(0)\backslash S$ must eventually
be contained in some $U_{l}$ we see that $\left\{  L_{i}\right\}  $ converges
uniformly on $K$. The $C^{m}$ convergence also implies that each of these
limiting immersed submanifolds satisfies the Hamiltonian stationary equation
\eqref{hseq}, since by Proposition \ref{L=M} each graph satisfies the
\eqref{hseq}. Now, take a diagonal sequence $\{L_{i}\}$ to get a sequence
which converges on each open set $U_{l}$ in the $C^{m}$ topology to a
varifold, or vacates every $U_{l}$. By the definition of this limit, the
$n$-varifolds must be nested. In particular, the limit will be nonempty unless
a subsequence satisfies (as $L_{i}$ is connected) $L_{i}\subset B_{\tau_{l}%
}(x_{j})$ for arbitrary small $\tau_{l}$ and some point $x_{j}\in S$. We are
assuming that $\left\{  L_{i}\right\}  $ does not converge to a point, so we
conclude that the limit is a nonempty varifold on $B_{R_{0}}(0)\backslash S,$
and we call its support $L$.

So, the immersed submanifold $L$ is Hamiltonian stationary and Lagrangian in
$B_{R_{0}}\backslash S$, because $\theta_{i}$ is harmonic on each $L_{i}(y)$,
moreover $\left\Vert {H}\right\Vert _{L^{n}(U_{l})}\leq C$ for all $l,$ so
$H\in L^{n}\left(  B_{R_{0}}(0)\backslash S\right)  $ from the smooth
convergence. By Proposition \ref{points} and Theorem \ref{N} with $k=0$ and
noticing $\overline{L}\backslash L\subset S$, $\overline{L}$ is Hamiltonian
stationary in $\mathbb{R}^{2n}$.

The argument above works when $S$ is empty as well, in that case $U_{0}%
=U_{l}=B_{R_{0}}(0)$.


Finally, we show that $\overline{L}$ is connected. Suppose that $\overline{L}$
is disconnected. As $\overline{L}$ is closed and bounded, each component is
compact, and there will be a smooth function $\psi$ on $B_{R_{0}}(0)$ such
that
\[
\overline{L}\subset\psi^{-1}\left(  0\right)  \cup\psi^{-1}\left(  1\right)
\]
with nontrivial intersection in both level sets. $\ $\ Now take a sequence of
points $p_{i}\in$ $L_{i}$ such that $p_{i}\rightarrow p\in\bar{L}\cap\psi
^{-1}\left(  0\right)  $ and $q_{i}\in$ $L_{i}$ such that $q_{i}\rightarrow
q\in\bar{L}\cap\psi^{-1}\left(  1\right)  .$ There is a path
\begin{align*}
\gamma_{i\text{ }}  &  :[0,1]\rightarrow L_{i}\\
\gamma_{i\text{ }}(0)  &  =p_{i}\\
\gamma_{i\text{ }}(0)  &  =q_{i}.
\end{align*}
as $L_{i}$ is connected. For all values $\sigma\in\lbrack0,1]$ there will a
value $t_{i}\left(  s\right)  $ such that $\psi(\gamma_{i\text{ }}%
(t_{i}))=\sigma$. In particular, for each $\sigma\in\lbrack\frac{1}{3}%
,\frac{2}{3}]$, there is a sequence of points $z_{i}(\sigma)\in L_{i}$ with
$\psi(z_{i})=\sigma.$ \ There are clearly infinitely many sequences converging
to different limit points (the continuous function $\psi$ distinguishes the
limit points), thus for some value $\sigma$ we can choose a limit point
$z_{i}\rightarrow z$ with $z$ not in $S$, as $S$ is finite, and $\psi
(z)\in\lbrack\frac{1}{3},\frac{2}{3}]$. \ Now $z$ has a positive distance
$d_{0}$ to the finite set $S$ of singular points, so we conclude that for
$\tau_{l}<<d_{0}$ the $L_{i}$ converge smoothly near $z,$ thus $z\in\bar{L}$
which contradicts that $\psi(z)\in\lbrack\frac{1}{3},\frac{2}{3}].$
\end{proof}

\end{document}